\documentclass[times]{nlaauth}

\usepackage{moreverb}
\usepackage{relsize}
\usepackage{amsthm}

\usepackage{subcaption}
\newcommand{\matr}[1]{\mathlarger{\mathbf{#1}}}
\newcommand{\Ptimat}{\matr{\tilde{P}}}
\newcommand{\Phatmat}{\matr{\hat{P}}}
\newcommand{\dt}{\Delta t}
\newcommand{\dx}{\Delta x}

\newcommand\BibTeX{{\rmfamily B\kern-.05em \textsc{i\kern-.025em b}\kern-.08em
T\kern-.1667em\lower.7ex\hbox{E}\kern-.125emX}}

\newtheorem{lemma}{Lemma}
\newtheorem{theorem}{Theorem}
\newtheorem{remark}{Remark}

\graphicspath{{images/}}

\begin{document}

\title{Asymptotic convergence of the parallel full approximation scheme in space and time for linear problems}

\author{Matthias Bolten\affil{1}, Dieter Moser\affil{2}, Robert Speck\affil{2}\corrauth  }

\address{
\affilnum{1}Department of Mathematics, Universit\"at Kassel, Germany. \break
\affilnum{2} J\"ulich Supercomputing Centre, Forschungszentrum J\"ulich GmbH, Germany.}

\corraddr{E-mail: r.speck@fz-juelich.de}

\begin{abstract}
For time-dependent partial differential equations, parallel-in-time integration using the ``parallel full approximation scheme in space and time'' (PFASST) is a promising way to accelerate existing space-parallel approaches beyond their scaling limits. 
Inspired by the classical Parareal method and multigrid ideas, PFASST allows to integrate multiple time-steps simultaneously using a space-time hierarchy of spectral deferred correction sweeps. 
While many use cases and benchmarks exist, a solid and reliable mathematical foundation is still missing.
Very recently, however, PFASST for linear problems has been identified as a multigrid method.
in this paper, we will use this multigrid formulation and in particular PFASST's iteration matrix to show that in the non-stiff as well as in the stiff limit PFASST indeed is a convergent iterative method.
We will provide upper bounds for the spectral radius of the iteration matrix and investigate how PFASST performs for increasing numbers of parallel time-steps.
Finally, we will demonstrate that the results obtained here indeed relate to actual PFASST runs.
\end{abstract}


\maketitle

\section{Introduction}

With the advent of supercomputing architectures featuring millions of processing units,  classical parallelization techniques used to accelerate the solution of discretized partial differential equations face new challenges. 
For fixed-size problems, communication starts to dominate eventually, when only small portions of data are left for computation on each unit.
This ``trap of strong scaling'' leads to severe and inevitable upper limits for speedup obtainable with parallelization in space, leaving large parts of extreme scale supercomputers unexploited.
If weak scaling is the target, this may not be an issue, but for time-dependent problems stability considerations often lead to an increase in the number of time-steps as the problem is refined in space.
This is not mitigated by spatial parallelization alone, yielding the ``trap of weak scaling''.
Thus, the challenges arising from the extreme levels of parallelism required by today's and future high-performance computing systems mandates the development of new numerical methods that feature a maximum degree of concurrency.

\bigskip

For time-dependent problems, in particular for time-dependent partial differential equations, approaches for the parallelization along the temporal dimension have become increasingly popular over the last years.
In his seminal work in 2015 Gander lists over $25$ approaches to parallelize the seemingly serial process of time integration~\cite{Gander2015}. 
In particular, the invention of the Parareal method in 2001~\cite{LionsEtAl2001} alone sparked a multitude of new developments in this area. 
This ``parallelization across the step'' approach allows to integrate many time-steps simultaneously.
This can work on top of already existing parallelization strategies in space.
The idea is to derive a coarser, less expensive time-integration scheme for the problem at hand and use this so-called coarse propagator to quickly and serially propagate information forward in time.
The original integrator, in this context often called the fine propagator, is then used in parallel-in-time using the initial values the coarse scheme provided.
This cycle of fine and coarse, parallel and serial time-stepping is repeated and upon convergence, Parareal is as accurate as the fine propagator run in serial. 
This way, the costs of the expensive fine scheme are distributed, while the serial-in-time part is kept small using a cheap propagator.
This predictor-corrector approach, being easy to implement and easy to apply, has been analyzed extensively.
It has been identified as a multiple shooting method or as an FAS multigrid scheme~\cite{GanderVandewalle2007_SISC} and convergence has been proven under various conditions, see e.g.\ \cite{GanderVandewalle2007_SISC, GanderVandewalle2007, Gander2008, GanderHairer2008, StaffRonquist2005}.

\bigskip
Yet, a key drawback of Parareal is the severe upper bound on parallel efficiency.
If $K$ iterations are required for convergence, the efficiency is bounded by $1/K$.
Perfect linear speedup cannot be expected due to the serial coarse propagator, but efficiencies of a few percent are also not desirable.
Therefore, many researchers started to enhance the Parareal idea with the goal of loosening this harsh bound on parallel efficiency.
One idea is to replace the fine and the coarse propagators by iterative solvers and coupling their ``inner'' iteration with the ``outer'' Parareal iteration.
A first step in this direction was done in~\cite{Minion2010}, where spectral deferred correction methods (SDC, see~\cite{DuttEtAl2000}) were used within Parareal.
This led to the ``parallel full approximation scheme in space and time'' (PFASST), which augmented this approach by ideas from non-linear multigrid methods~\cite{EmmettMinion2012,EmmettMinion2014_DDM}.
In these original papers from 2012 and 2014, the PFASST algorithm was introduced, its implementation was discussed and it was applied to first problems.
In the following years, PFASST has been applied to more and more problems and coupled to different space-parallel solvers, ranging from a Barnes-Hut tree code to geometric multigrid, see~\cite{SpeckEtAl2012,SpeckEtAl2014_Parco,SpeckEtAl2014_DDM2012,MinionEtAl2015}.
Together with spatial parallelization, PFASST was demonstrated to run and scale on up to $458\mathord{,}752$ cores of an IBM Blue Gene/Q installation.
Yet, while applications, implementation and improvements are discussed frequently, a solid and reliable convergence theory is still missing.
While for Parareal many results exist and provide a profound basis for a deep understanding of this algorithm, this is by far not the case for PFASST.
Very recently, however, PFASST for linear problems was identified as a multigrid method in~\cite{doi:10.1002/nla.2110,Moser2017PhD} and the definition of its iteration matrix yielded a new understanding of the algorithm's components and their mechanics.
This understanding allows to analyze the method using the established Local Fourier Analysis (LFA) technique. 
LFA has been introduced to study smoothers in~\cite{BraMulti1977}, later it was extended to study the whole multigrid algorithm~\cite{STMultigrid1982} and since then it has become a standard tool for the analysis of multigrid. 
For a detailed introduction see~\cite{TOSMultigrid2001,WJPractical2005}. 
In the context of space-time multigrid the results obtained using plain LFA are less meaningful because of the non-normality due to the discretization of the time domain. 
To overcome this limitation the semi-algebraic mode analysis (SAMA) has been introduced in \cite{FMGeneralized2015}. 
A kindred idea has been used to analyze PFASST in~\cite{doi:10.1002/nla.2110,Moser2017PhD}.
Although this careful block Fourier mode analysis already revealed many interesting features and also limitations, a rigorous proof of convergence has not been provided so far.

\bigskip

In this paper, we will use the multigrid formulation of PFASST for linear problems and in particular the iteration matrix to show that in the non-stiff as well as in the stiff limit PFASST indeed is a convergent iterative method.
We will provide upper bounds for the spectral radius of the iteration matrix and show that under certain assumptions, PFASST also satisfies the approximation property of standard multigrid theory. 
In contrast, the smoothing property does not hold, but we will state a modified smoother which allows to satisfy also this property.
We will further investigate how PFASST performs for increasing numbers of parallel time-steps.
Finally, we will demonstrate that the results obtained here indeed relate to actual PFASST runs.
We start with a brief summary of the results found in~\cite{doi:10.1002/nla.2110}, describing PFASST as a multigrid method.

\section{A multigrid perspective on PFASST}

We focus on linear, autonomous systems of ordinary differential equations (ODEs) with
\begin{align}
  \begin{split}\label{eq:ode}
    u_t(t) &= \matr{S}u(t)\ \text{for}\ t\in[0,T],\\
    u(0) &= u_0
  \end{split}
\end{align}
with $u(t) \in \mathbb{C}^N$, $T>0$, initial value $u_0\in\mathbb{C}^N$ and ``spatial'' matrix $\matr{S}\in\mathbb{C}^{N\times N}$, stemming from e.g.\ a spatial discretization of a partial differential equation (PDE).
Examples include the heat or the advection equation, but also the wave equation and other types of linear PDEs and ODEs.

\subsection{The collocation problem and SDC}

We decompose the time interval into $L$ subintervals $[t_l,t_{l+1}]$, $l=0, ..., L-1$ and rewrite the ODE for such a time-step in Picard formulation as
\begin{align*}
  u(t) = u_l + \int_{t_l}^t \matr{S}u(s)ds,\ t\in[t_l,t_{l+1}],
\end{align*}
where $u_l$ is the initial condition for this time-step, e.g.\ coming from a time-stepping scheme.
Introducing $M$ quadrature nodes $\tau_1,...,\tau_M$ with $t_l \le \tau_1 < ... < \tau_M = t_{l+1}$, we can approximate the integrals from $t_l$ to these nodes $\tau_m$ using spectral quadrature like Gau\ss-Radau or Gau\ss-Lobatto quadrature, such that
\begin{align*}
  u_m = u_l + \dt\sum_{j=1}^M q_{m,j} \matr{S}u_j \approx u_l + \int_{t_l}^{\tau_m} \matr{S}u(s)ds,\ \text{for}\ m=1,...,M,
\end{align*}
where $u_m \approx u(\tau_m)$, $\dt = t_{l+1}-t_l$ and $q_{m,j}$ represent the quadrature weights for the interval $[t_l,\tau_m]$ with
\begin{equation*}
	q_{m,j} := \int_{t_l}^{\tau_m} \ell_j(s)ds, \quad m,j=1,\ldots,M,
\end{equation*}
where $\ell_j$ are the Lagrange polynomials to the points $\tau_m$.
Note that for the quadrature rule on each subinterval $[\tau_m,\tau_{m+1}]$, $m=1, ..., M-1$ all collocation nodes are taken into account, even if they do not belong to the subinterval under consideration.
Combining this into one set of linear equations yields
\begin{align}\label{eq:coll_orig}
  U = U_l + \dt \left(\matr{Q}\otimes\matr{S}\right) U\quad \text{or}\quad \left(\matr{I}_{MN} - \dt\matr{Q}\otimes\matr{S}\right)U = U_l
\end{align}
for vectors $U = (u_1, ..., u_M)^T, U_l = (u_l, ..., u_l)^T\in\mathbb{C}^{MN}$ and quadrature matrix $\matr{Q} = (q_{m,j})\in\mathbb{R}^{M\times M}$.
This is the so-called ``collocation problem'' and it is equivalent to a fully implicit Runge-Kutta method.
Before we proceed with describing the solution strategy for this problem, we slightly change the notation:
Instead of working with the term $\dt\matr{Q}\otimes\matr{S}$, we introduce the ``CFL number'' $\mu$ (sometimes called the ``discrete dispersion relation number'') to absorb the time-step size $\dt$, problem-specific parameters like diffusion coefficients as well as the spatial mesh size $\dx$, if applicable.
We write
\begin{align*}
  \dt\matr{S} = \mu\matr{A},
\end{align*}
where the matrix $\matr{A}$ is the normalized description of the spatial problem or system of ODEs.
For example, for the heat and the advection equation, the parameter $\mu$ is defined by 
\begin{align*}
  \mu_\mathrm{diff} = \nu\frac{\dt}{\dx^2},\quad \mu_\mathrm{adv} = c\frac{\dt}{\dx}
\end{align*}
with diffusion coefficient $\nu$ and advection speed $c$.
Then, Equation~\eqref{eq:coll_orig} reads
\begin{align}\label{eq:coll}
  \left(\matr{I}_{MN} - \mu\matr{Q}\otimes\matr{A}\right)U = U_l
\end{align}
and we will use this form for the remainder of this paper.

\bigskip

This system of equations is dense and a direct solution is not advisable, in particular if the right-hand side of the ODE is non-linear.
While the standard way of solving this is a simplified Newton approach~\cite{opac-b1130632}, the more recent development of spectral deferred correction methods (SDC, see~\cite{DuttEtAl2000}) provides an interesting and very flexible alternative.
In order to present this approach, we follow the idea of preconditioned Picard iteration as found e.g.\ in~\cite{HuangEtAl2006, Weiser2014, RuprechtSpeck2016}.
The key idea here is to provide a flexible preconditioner based on a simpler quadrature rule for the integrals.
More precisely, the iteration $k$ is given by
\begin{align*}
  U^{k+1} = U^{k} + \left(\matr{I}_{MN} - \mu\matr{Q}_\Delta\otimes\matr{A}\right)^{-1}\left(U_l - \left(\matr{I}_{MN} - \mu\matr{Q}\otimes\matr{A}\right)U^k\right),\ \text{for}\ k = 0, ... K,
\end{align*}
with $K\in\mathbb{N}$ and where the matrix $\matr{Q}_\Delta\in\mathbb{R}^{M\times M}$ gathers the weights of this simpler quadrature rule.
Examples are the implicit right-hand side rule or the explicit left-hand side rule, both yielding lower triangular matrices, which make the inversion of the preconditioner straightforward using simple forward substitution.
More recently, Weiser~\cite{Weiser2014} defined $\matr{Q}_\Delta = \matr{U^T}$ for $\matr{Q}^T = \matr{L}\matr{U}$ and showed superior convergence properties of SDC for stiff problems. 
This approach has become the de-facto standard for SDC preconditioning and is colloquially known as St.~Martin's or LU trick.
Now, for each time-step, SDC can be used to generate an approximate solution of the collocation problem~\eqref{eq:coll}. 
As soon as SDC has converged (e.g.\ the residual of the collocation problem is smaller than a prescribed threshold), the solution at $\tau_M$ is used as initial condition for the next time-step.
In order to parallelize this, the ``parallel full approximation scheme in space and time'' (PFASST, see~\cite{EmmettMinion2012}) makes use of a space-time hierarchy of SDC iterations (called ``sweeps'' in this context), using the coarsest level to propagate information quickly forward in time. This way, multiple time-steps can be integrated simultaneously, where on each local time interval SDC sweeps are used to approximate the collocation problem.
We follow~\cite{doi:10.1002/nla.2110,Moser2017PhD} to describe the PFASST algorithm more formally.

\subsection{The composite collocation problem and PFASST}

The problem PFASST is trying to solve is the so called ``composite collocation problem'' for $L\in\mathbb{N}$ time-steps with
\begin{align*}
  \begin{pmatrix}
    \matr{I}_{MN} - \mu\matr{Q}\otimes\matr{A} \\
    -\matr{H} & \matr{I}_{MN} - \mu\matr{Q}\otimes\matr{A} \\
     & \ddots & \ddots \\
    & & -\matr{H} & \matr{I}_{MN} - \mu\matr{Q}\otimes\matr{A}
  \end{pmatrix}
  \begin{pmatrix}
    U_1\\
    U_2\\
    \vdots\\
    U_L
  \end{pmatrix} = 
  \begin{pmatrix}
    U_0\\
    0\\
    \vdots\\
    0
  \end{pmatrix}.
\end{align*}
The system matrix consists of $L$ collocation problems on the block diagonal and a transfer matrix $\matr{H} = \matr{N}\otimes\matr{I}_N\in\mathbb{R}^{MN\times MN}$ on the lower diagonal, which takes the last value of each time-step and makes it available as initial condition for the next one.
With the nodes we choose here, $\matr{N}$ is simply given by
\begin{align*}
  \matr{N} =  
    \begin{pmatrix}
      0 & 0 & \cdots & 1 \\
      0 & 0 & \cdots & 1 \\
       \vdots & \vdots & & \vdots \\
      0 & 0 & \cdots & 1 \\
    \end{pmatrix} \in \mathbb{R}^{M\times M}.
\end{align*}
For collocation nodes with $\tau_M < t_{l+1}$, i.e. with the last node $\tau_M$ not being identical with the subinterval boundary point $t_{l+1}$, the matrix $\matr{N}$ would contain an extrapolation rule for the solution value at $t_{l+1}$ in each line.
Note that instead of extrapolation the collocation formulation could be used as well.
More compactly and more conveniently, the composite collocation problem can be written as
\begin{align*}
  \matr{C}\vec{U} = \vec{U}_0
\end{align*}
with space-time-collocation vectors $\vec{U} = (U_1, ..., U_L)^T, \vec{U}_0 = (U_0, 0, ..., 0)^T\in\mathbb{R}^{LMN}$ and system matrix $\matr{C} = \matr{I}_{LMN} - \mu \matr{I}_L\otimes\matr{Q} \otimes \matr{A} - \matr{E}\otimes\matr{H}\in\mathbb{C}^{LMN\times LMN}$, where the matrix $\matr{E}\in\mathbb{R}^{L\times L}$ simply has ones on the first lower subdiagonal and zeros elsewhere.

\bigskip

The key idea of PFASST is to solve the composite collocation problem using a multigrid scheme.
If the right-hand side of the ODE~\eqref{eq:ode} is non-linear, a non-linear FAS multigrid is used.
Although our analysis is focused on linear problems, we use FAS terminology to formulate PFASST to remain consistent with the literature.
Also, we limit ourselves to a two-level scheme in order to keep the notation as simple as possible. 
Three components are needed to describe the multigrid scheme used to solve the composite collocation problem: (1) a smoother on the fine level, (2) a solver on the coarse level and (3) level transfer operators.
In order to obtain parallelism, the smoother we choose is an approximative block Jacobi smoother, where the entries on the lower subdiagonal are omitted.
The idea is to use SDC within each time-step (this is why it is an ``approximative'' Jacobi iteration), but omit the transfer matrices $\matr{H}$ on the lower diagonal.
In detail, the smoother is defined by the preconditioner
\begin{align*}
  \Phatmat = 
  \begin{pmatrix}
    \matr{I}_{MN} - \mu\matr{Q}_\Delta\otimes\matr{A} \\
     & \matr{I}_{MN} - \mu\matr{Q}_\Delta\otimes\matr{A} \\
     & & \ddots \\
    & & & \matr{I}_{MN} - \mu\matr{Q}_\Delta\otimes\matr{A}
  \end{pmatrix}
\end{align*}
or, more compactly, $\Phatmat = \matr{I}_{LMN} - \mu \matr{I}_L\otimes\matr{Q}_\Delta \otimes \matr{A}$.
Inversion of this matrix can be done on all $L$ time-steps simultaneously, leading to $L$ decoupled SDC sweeps.
Note that typically this is done only once or twice on the fine level.
In contrast, the solver on the coarse level is given by an approximative block Gau\ss-Seidel preconditioner.
Here, SDC is used for each time-step, but the transfer matrix $\matr{H}$ is included.
This yields for the preconditioner
\begin{align*}
  \Ptimat = 
  \begin{pmatrix}
    \matr{I}_{MN} - \mu\matr{Q}_\Delta\otimes\matr{A} \\
    -\matr{H} & \matr{I}_{MN} - \mu\matr{Q}_\Delta\otimes\matr{A} \\
     & \ddots & \ddots \\
    & & -\matr{H} & \matr{I}_{MN} - \mu\matr{Q}_\Delta\otimes\matr{A}
  \end{pmatrix}
\end{align*}
or $\Ptimat = \matr{I}_{LMN} - \mu\matr{I}_L\otimes \matr{Q}_\Delta \otimes \matr{A} - \matr{E}\otimes\matr{H}$.
Inversion of this preconditioner is inherently serial, but the goal is to keep this serial part as small as possible by applying it on the coarse level only, just as the Parareal method does.
Thus, we need coarsening strategies in place to reduce the costs on the coarser levels~\cite{Speck2015}.
To this end, we introduce block-wise restriction and interpolation $\matr{T}_F^C$ and $\matr{T}_C^F$, which coarsen the problem in space and reduce the number of quadrature nodes but do not coarsen in time, i.e., the number of time steps is not reduced.
We note that the latter is also possible in this formal notation, but so far no PFASST implementation is working with this.
Also, the theory presented here makes indeed use of the fact that coarsening across time-steps is not applied.
A first discussion on this topic can be found in~\cite{Moser2017PhD}.
Let $\tilde{N}\in\mathbb{N}$ be the number of degrees of freedom on the coarse level and $\tilde{M}\in\mathbb{N}$ the number of collocation nodes on the coarse level. 
Restriction and interpolation operators are then given by
\begin{align*}
  \matr{T}_F^C &= \matr{I}_L\otimes \matr{T}_{F,Q}^C\otimes\matr{T}_{F,A}^C\in\mathbb{R}^{L\tilde{M}\tilde{N}\times LMN},\\
  \matr{T}_C^F &= \matr{I}_L\otimes \matr{T}_{C,Q}^F\otimes\matr{T}_{C,A}^F\in\mathbb{R}^{LMN\times L\tilde{M}\tilde{N}},
\end{align*}
where the matrices $\matr{T}_{F,Q}^C$ and $\matr{T}_{C,Q}^F$ represent restriction and interpolation on the quadrature nodes, while $\matr{T}_{F,A}^C$ and $\matr{T}_{C,A}^F$ operate on spatial degrees-of-freedom.
Within PFASST, these operators are typically standard Lagrangian-based restriction and interpolation.
We use the tilde symbol to denote matrices on the coarse level, so that the approximative block Gau\ss-Seidel preconditioner is actually given by
\begin{align*}
  \Ptimat = \matr{I}_{L\tilde{M}\tilde{N}} - \mu\matr{I}_L\otimes \matr{\tilde{Q}}_\Delta \otimes \matr{\tilde{A}} - \matr{E}\otimes\matr{\tilde{H}}\in\mathbb{R}^{L\tilde{M}\tilde{N}\times L\tilde{M}\tilde{N}},
\end{align*}
where $\matr{\tilde{H}} = \matr{\tilde{N}}\otimes\matr{I}_{\tilde{N}}\in\mathbb{R}^{\tilde{M}\tilde{N}\times\tilde{M}\tilde{N}}$.
Note that this preconditioner is typically applied only once or twice on the coarse level, too.
In addition, the composite collocation problem has to be modified on the coarse level. 
This is done by the $\tau$-correction of the FAS scheme and we refer to~\cite{doi:10.1002/nla.2110,Moser2017PhD} for details on this, since the actual formulation does not matter here.
We now have all ingredients for one iteration of the two-level version of PFASST using post-smoothing: 
\begin{enumerate}  
  \item restriction to the coarse level including the formulation of the $\tau$-correction,
  \item serial approximative block Gau\ss-Seidel iteration on the modified composite collocation problem on the coarse level,
  \item coarse-grid correction of the fine-level values,
  \item smoothing of the composite collocation problem on the fine level using parallel approximative block Jacobi iteration.
\end{enumerate}
Thus, one iteration of PFASST can simply be written as
\begin{align*}
  \vec{U}^{k+1/2} &= \vec{U}^{k} + \matr{T}_C^F\Ptimat^{-1}\matr{T}_F^C\left(\vec{U}_0 - \matr{C}\vec{U}^{k}\right),\\
  \vec{U}^{k+1} &= \vec{U}^{k+1/2} + \Phatmat^{-1}\left(\vec{U}_0 - \matr{C}\vec{U}^{k+1/2}\right).
\end{align*}
Beside this rather convenient and compact form, this formulation has the great advantage of providing the iteration matrix of PFASST, which paves the way to a comprehensive analysis.

\subsection{Overview and notation}

In the following, we summarize the results described above and state the iteration matrix of PFASST.

\begin{theorem}
  Let $\matr{T}_F^C$ and $\matr{T}_C^F$ be block-wise defined transfer operators, which treat the subintervals $[t_l,t_{l+1}]$, $l=0, ..., L-1$ independently from each other (i.e.\ which do not coarsen or refine across subinterval boundaries).
  For a CFL number $\mu>0$ we define the composite collocation problem as
  \begin{align*}
     \matr{C} = \matr{I}_{LMN} - \mu \matr{I}_L\otimes\matr{Q} \otimes \matr{A} - \matr{E}\otimes\matr{H},
  \end{align*}
  with collocation matrix $\matr{Q}$, spatial matrix $\matr{A}$ (for more details see~\eqref{eq:coll}) and 
  \begin{align*}
      \matr{H}=\matr{N} \otimes \matr{I}_{N}\in \mathbb{R}^{NM}\quad\text{with}\quad\matr{N} &=  
        \begin{pmatrix}
          0 & 0 & \cdots & 1 \\
          0 & 0 & \cdots & 1 \\
           \vdots & \vdots & & \vdots \\
          0 & 0 & \cdots & 1 \\
        \end{pmatrix} \in \mathbb{R}^{M\times M}
  \end{align*}
  and $\matr{E} \in \mathbb{R}^{L\times L}$ being a matrix which has ones on the first subdiagonal and zeros elsewhere.
  We further define by $\matr{\hat{P}}$ the approximative block Jacobi preconditioner on the fine level and by $\matr{\tilde{P}}$ the approximative block Gauss-Seidel preconditioner on the coarse level (the tilde symbol always indicates the coarse level operators), i.e.
  \begin{align*}
    \Ptimat &= \matr{I}_{L\tilde{M}\tilde{N}} - \mu\matr{I}_L\otimes \matr{\tilde{Q}}_\Delta \otimes \matr{\tilde{A}} - \matr{E}\otimes\matr{\tilde{H}},\\
    \Phatmat &= \matr{I}_{LMN} - \mu \matr{I}_L\otimes\matr{Q}_\Delta \otimes \matr{A} 
  \end{align*}
  where $\matr{Q}_\Delta$ corresponds to a simple quadrature rule.
  For given $\matr{H}$ and $\matr{\tilde{H}}$ we require that the restriction operator $\matr{T}_F^C$ satisfies $(\matr{E}\otimes\matr{\tilde{H}})\matr{T}_F^C  = \matr{T}_F^C (\matr{E}\otimes\matr{H})$.
  Then, the PFASST iteration matrix is given by the product of the smoother's and the coarse-grid correction's iteration matrix with
  \begin{align*}
    \matr{T}_{\mathrm{PFASST}}  = \matr{T}_{\mathrm{S}}\matr{T}_{\mathrm{CGC}}=\left( \matr{I}_{LMN} - \Phatmat^{-1}\matr{C} \right) \left( \matr{I}_{LMN} - \matr{T}_C^F \Ptimat^{-1} \matr{T}_F^C \matr{C}\right).
  \end{align*}
  We note that we use the same CFL number $\mu$ for both coarse and fine level and absorb constant factors between the actual CFL numbers of the coarse and the fine problem into the operators $\matr{A}$ and $\matr{\tilde{A}}$.
  \label{th:pfasst_in_matrix_form}
\end{theorem}

\begin{proof}
  This is taken from~\cite{doi:10.1002/nla.2110,Moser2017PhD} and a much more detailed derivation and discussion can be found there.
\end{proof}


Note that we assume here and in the following that the inverses of $\Phatmat$ and $\Ptimat$ both exist.
This corresponds to the fact that time-stepping via $\matr{Q}_\Delta$ is possible on each subinterval.

In what follows, we are interested in the behavior of PFASST's iteration matrix for the non-stiff as well as for the stiff limit. 
More precisely, we will look at the case where the CFL number $\mu$ either goes to zero or to infinity.
While the first case represents the analysis for smaller and smaller time-steps, the second one covers scenarios where e.g. the mesh- or element-size $\dx$ goes to zero.
Alternatively, problem parameters like the diffusion coefficient or the advection speed could become very small or very large, while $\dt$ and $\dx$ are fixed.

\section{The non-stiff limit}\label{sec:mu_to_0}

This section is split into three parts.
We look at the iteration matrices of the smoother and the coarse-grid correction separately and then analyze the full iteration matrix of PFASST.
While for the smoother we introduce the main idea behind the asymptotic convergence analysis, the analysis of the coarse-grid correction is dominated by the restriction and interpolation operators.
For PFASST's full iteration matrix we then combine both results in a straightforward manner.

\subsection{The smoother}

We first consider the iteration matrix $\matr{T}_{\mathrm{S}}$ of the smoother with
\begin{align*}
  \matr{T}_{\mathrm{S}} &= \matr{I}_{LMN} - \Phatmat^{-1}\matr{C}\\
   &= \matr{I}_{LMN} - \left(\matr{I}_{LMN} - \mu \matr{I}_L\otimes\matr{Q}_\Delta \otimes \matr{A} \right)^{-1}\left(\matr{I}_{LMN} - \mu \matr{I}_L\otimes\matr{Q} \otimes \matr{A} - \matr{E}\otimes\matr{H}\right).
\end{align*}
We write $\matr{T}_{\mathrm{S}} = \matr{T}_{\mathrm{S}}(\mu)$, so that
\begin{align*}
    \matr{T}_{\mathrm{S}}(0) = \matr{E}\otimes\matr{H}.
\end{align*}
Therefore, we have
\begin{align}
  \begin{split}\label{eq:diff_smoo}
    \matr{T}_{\mathrm{S}}(\mu) - \matr{T}_{\mathrm{S}}(0) &= \Phatmat^{-1}\left(\Phatmat-\matr{C} - \Phatmat\left(\matr{E}\otimes\matr{H}\right)\right)\\
    &= \mu\Phatmat^{-1}\left(\matr{I}_L\otimes\left(\matr{Q}-\matr{Q}_\Delta\right)\otimes\matr{A} +\matr{E}\otimes\matr{Q}_\Delta\matr{N}\otimes\matr{A}\right).
  \end{split}
\end{align}
Moreover, if $\mu$ is smaller than a given value $\mu^*_\mathrm{S}>0$, the norm of $\Phatmat^{-1}$ can be bounded by 
\begin{align}\label{eq:smoo_mu_bound_to0}
    \left\lVert\Phatmat^{-1}\right\rVert  = \left\lVert\Phatmat^{-1}(\mu)\right\rVert = \left\lVert\left(\matr{I}_{LMN} - \mu \matr{I}_L\otimes\matr{Q}_\Delta \otimes \matr{A} \right)^{-1}\right\rVert \le c_1(\mu^*_{\mathrm{S},0}) = c_1,
\end{align}
for a constant $c_1(\mu^*_{\mathrm{S},0})$ independent of $\mu$, since the function $\mu\mapsto\left\lVert\Phatmat^{-1}(\mu)\right\rVert$ is continuous on the closed interval $[0, \mu^*_{\mathrm{S},0}]$.
The norm $\left\lVert.\right\rVert$ can be any induced matrix norm unless stated otherwise.
Together with~\eqref{eq:diff_smoo} we obtain
\begin{align*}
    \left\lVert\matr{T}_{\mathrm{S}}(\mu) - \matr{T}_{\mathrm{S}}(0)\right\rVert \le c_2\mu
\end{align*}
since the last factor of~\eqref{eq:diff_smoo} does not depend on $\mu$.

Therefore, the matrix $\matr{T}_{\mathrm{S}}(\mu)$ converges to $\matr{T}_{\mathrm{S}}(0) = \matr{E} \otimes \matr{H}$ linearly as $\mu\rightarrow 0$, so that we can write
\begin{align}\label{eq:smoo_perturbed}
    \matr{T}_{\mathrm{S}}(\mu) = \matr{T}_{\mathrm{S}}(0)+ \mathcal{O}(\mu) =\matr{E} \otimes \matr{H} + \mathcal{O}(\mu).
\end{align}
This leads us to the following lemma:

\begin{lemma}\label{lem:smoother_to0}
    The smoother of PFASST converges for linear problems and Gau\ss-Radau nodes, if the CFL number $\mu$ is small enough. 
    More precisely, for $L$ time-steps the spectral radius of the smoother is bounded by
    \begin{align}\label{eq:smoother_est}
        \rho\left(\matr{T}_{\mathrm{S}}(\mu)\right) \le c\mu^\frac{1}{L}
    \end{align} 
    for a constant $c>0$ independent of $\mu$, if $\mu < \mu^*_{\mathrm{S},0}$ for a fixed value $\mu^*_{\mathrm{S},0} > 0$.
\end{lemma}

\begin{proof}
    The matrix $\matr{T}_{\mathrm{S}}(\mu)$ can be seen as a perturbation of $\matr{T}_{\mathrm{S}}(0)$, where the perturbation matrix $\matr{D}(\mu)$ is of the order of $\mathcal{O}(\mu)$.
    The eigenvalues of the unperturbed matrix $\matr{T}_{\mathrm{S}}(0)$ are all zero (since it only has entries strictly below the diagonal).
    With
    \begin{align*}
        \matr{T}_{\mathrm{S}}(0) = \matr{E} \otimes \matr{H} = \matr{E} \otimes \matr{N} \otimes \matr{I},
    \end{align*}
    its Jordan canonical form also consists for three parts.
    Obviously, the canonical form of $\matr{E}$ consists of a single Jordan block of size $L$ for the eigenvalue $0$, while $\matr{I}$ has $N$ Jordan blocks of size $1$, with $N$ being the number of degrees-of-freedom in space.
    For $M$ quadrature nodes, the canonical form of $\matr{N}$ consists of $M-1$ blocks of size $1$ for the eigenvalue $0$ and one block of size $1$ for the eigenvalue $1$. 
    Therefore, the canonical form of $\matr{T}_{\mathrm{S}}(0)$ consists of $MN$ blocks of size $L$ for the eigenvalue $0$.
    Since the perturbation matrix $\matr{D}(\mu)$ does not have a particular structure other than its linear dependence on $\mu$, the difference between the eigenvalues of the perturbed matrix $\matr{T}_{\mathrm{S}}(\mu)$ and the unperturbed matrix $\matr{T}_{\mathrm{S}}(0)$ is of the order of $\mathcal{O}(\mu^\frac{1}{\alpha})$, where $\alpha = L$ is the size of the largest Jordan block, see~\cite{Wilkinson}, p.~77ff.
    Thus, 
    \begin{align}
        \rho\left(\matr{T}_{\mathrm{S}}(\mu)\right) \le c\mu^\frac{1}{L}
    \end{align} 
    and especially $\rho\left(\matr{T}_{\mathrm{S}}(\mu)\right) < 1$ for $\mu$ small enough.
\end{proof}

Figure~\ref{fig:smoother_specrad} shows for Dahlquist's test problem $u_t = \lambda u$, $u(0) = 1$, that this estimate is severely over-pessimistic, if $L$ is small, but becomes rather accurate, if $L$ becomes larger. 
This does not change significantly when choosing an imaginary value for $\lambda$ of the test equation, as Figure~\ref{fig:smoother_specrad_imag_L64+256_M3LU} shows.
Note, however, that for these large numbers of time-steps $L$ the numerical computation of the spectral radius may be prone to rounding errors.
We refer to the discussion in~\cite{Moser2017PhD} for more details. 

\begin{figure}[t!]
  \centering
  \begin{subfigure}[b]{0.475\textwidth}
    \centering
    \includegraphics[width=\textwidth]{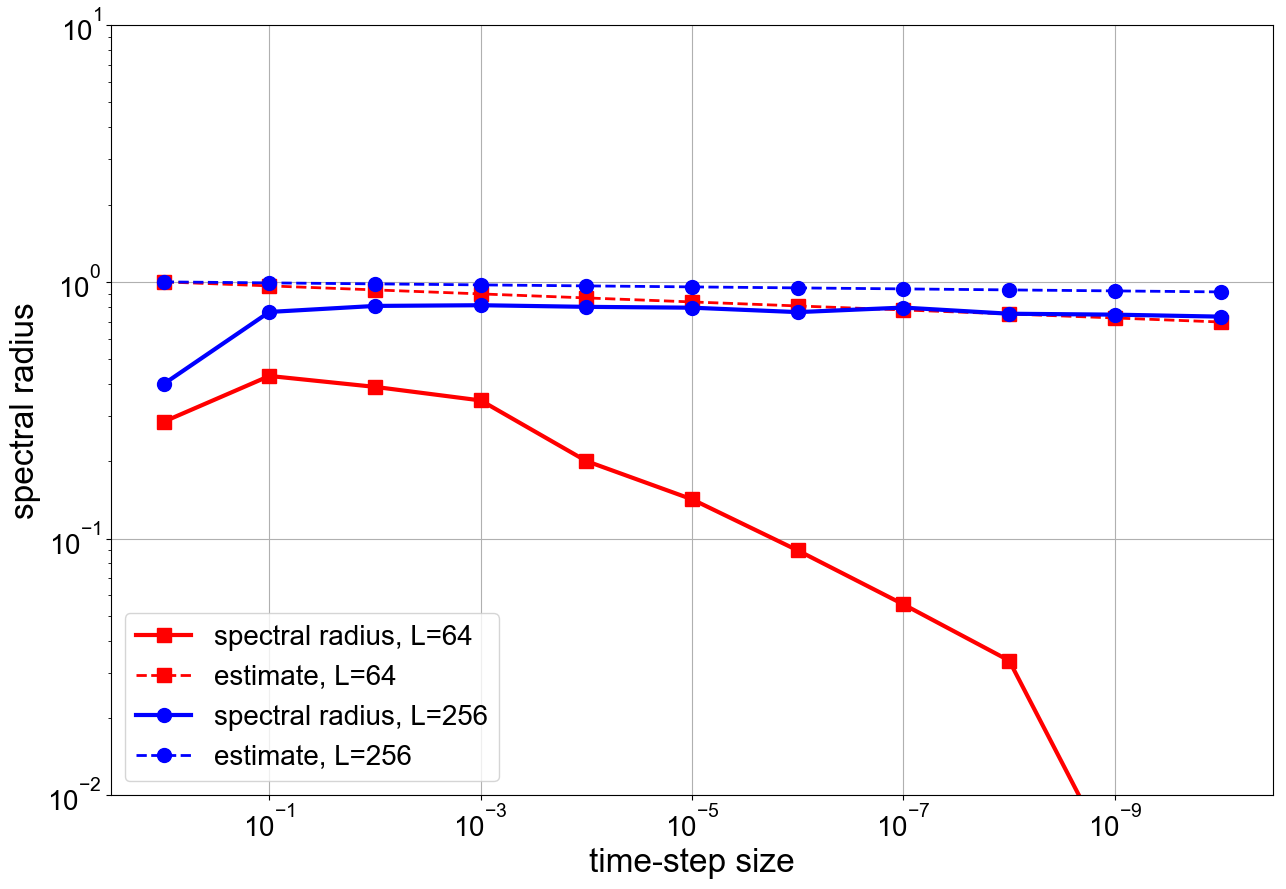}
    \caption{$\lambda=-1$}
    \label{fig:smoother_specrad_real_L64+256_M3LU}
  \end{subfigure}
  \begin{subfigure}[b]{0.475\textwidth}
    \centering
    \includegraphics[width=\textwidth]{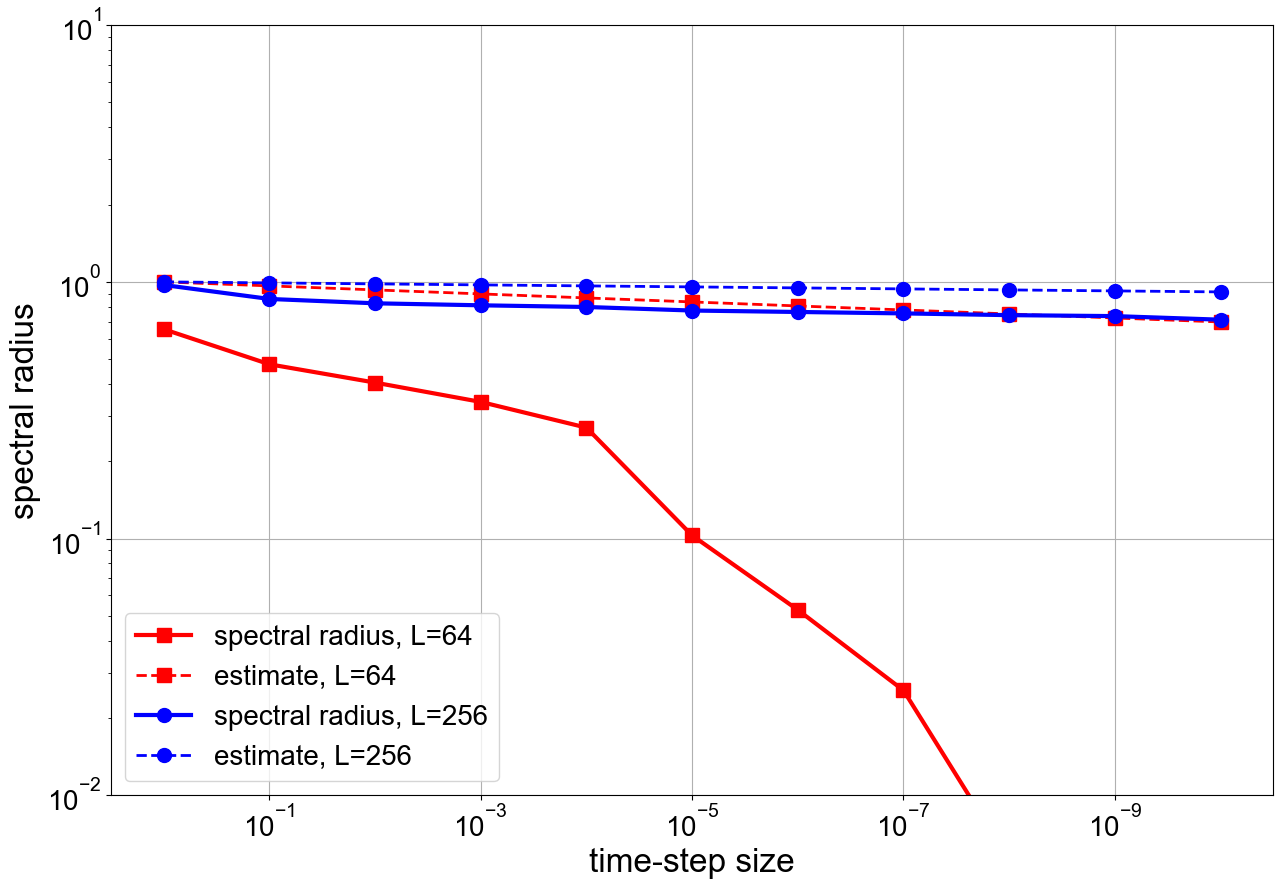}
    \caption{$\lambda=i$}
    \label{fig:smoother_specrad_imag_L64+256_M3LU}
  \end{subfigure}    
  \caption{Convergence of the spectral radius of the smoother vs.~the estimation~\eqref{eq:smoother_est} for Dahlquist's test problem $u_t = \lambda u$, $u(0) = 1$. Left: negative real value $\lambda = -1$, right: purely imaginary value $\lambda = i$. $3$ Gauss-Radau nodes are taken and the LU trick is used for $\matr{Q}_\Delta$, see~\cite{Weiser2014}.}
  \label{fig:smoother_specrad}
\end{figure}

\bigskip
Yet, while this lemma gives a proof of the convergence of the smoother, it cannot be used to estimate the speed of convergence.
The standard way of providing such an estimate would be to bound the norm of the iteration matrix $\matr{T}_{\mathrm{S}}(\mu)$ by something smaller than $1$. 
However, even in the limit $\mu\rightarrow0$ the norm of $\matr{T}_{\mathrm{S}}(\mu) = \matr{T}_{\mathrm{S}}(0)$ is still larger than or equal to $1$ in all feasible matrix norms.
Alternatively, we can look at the $k$th power of the iteration matrix, which corresponds to $k$ applications of the smoother.

\begin{remark}\label{rem:smoo_norm_bound}
    For all $k\in\mathbb{N}$ we have
    \begin{align*}
        \matr{T}_{\mathrm{S}}(\mu)^k = \left(\matr{T}_{\mathrm{S}}(0) + \mathcal{O}(\mu)\right)^k = \matr{T}_{\mathrm{S}}(0)^k + \mathcal{O}(\mu)
    \end{align*}
    for some perturbations of order $\mathcal{O}(\mu)$.
    The matrix $\matr{T}_{\mathrm{S}}(0) = \matr{E} \otimes \matr{H}$ is nilpotent with $\matr{T}_{\mathrm{S}}(0)^L = \matr{0}$, because $\matr{E}^L = 0$.
    Thus, for $k\ge L$,
    \begin{align*}
        \left\lVert\matr{T}_{\mathrm{S}}(\mu)^k\right\rVert \le \left\lVert\matr{T}_{\mathrm{S}}(0)^k\right\rVert + \mathcal{O}(\mu) = \mathcal{O}(\mu).
    \end{align*}
    This shows that for enough iterations of the smoother, the error is reduced in the order of $\mu$.
    However, for $k<L$, this is not true, since there is still a term larger than or equal to $1$ coming from $\matr{E} \otimes \matr{H}$.
    Of course, doing as many smoothing steps as we have time-steps is not feasible, so that this result is rather pathological.
\end{remark}

\subsection{The coarse-grid correction}\label{ssec:cgc_to0}

We now consider the iteration matrix $\matr{T}_{\mathrm{CGC}}$ of the coarse-grid correction with
\begin{align*}
  \matr{T}_{\mathrm{CGC}} = \matr{I}_{LMN} - \matr{T}_C^F \Ptimat^{-1} \matr{T}_F^C \matr{C}
  = \matr{I}_{LMN} - &\matr{T}_C^F\left(\matr{I}_{L\tilde{M}\tilde{N}} - \mu\matr{I}_L\otimes \matr{\tilde{Q}}_\Delta \otimes \matr{\tilde{A}} - \matr{E}\otimes\matr{\tilde{H}}\right)^{-1}\cdot\\
  &\matr{T}_F^C\left(\matr{I}_{LMN} - \mu \matr{I}_L\otimes\matr{Q} \otimes \matr{A} - \matr{E}\otimes\matr{H}\right)
\end{align*}
and write $\matr{T}_{\mathrm{CGC}} = \matr{T}_{\mathrm{CGC}}(\mu)$.
Then,
\begin{align*}
    \matr{T}_{\mathrm{CGC}}(0) &= \matr{I}_{LMN} - \matr{T}_C^F\left(\matr{I}_{L\tilde{M}\tilde{N}} - \matr{E}\otimes\matr{\tilde{H}}\right)^{-1}\matr{T}_F^C\left(\matr{I}_{LMN} - \matr{E}\otimes\matr{H}\right)\\
    &=\matr{I}_{LMN} - \matr{T}_C^F\matr{T}_F^C\left(\matr{I}_{LMN} - \matr{E}\otimes\matr{H}\right)^{-1}\left(\matr{I}_{LMN} - \matr{E}\otimes\matr{H}\right)\\
    &= \matr{I}_{LMN} - \matr{T}_C^F\matr{T}_F^C
\end{align*}
according to Theorem~\ref{th:pfasst_in_matrix_form}.
Therefore, we get
\begin{align*}
    \matr{T}_{\mathrm{CGC}}(\mu) - \matr{T}_{\mathrm{CGC}}(0) = \matr{T}_C^F\Ptimat^{-1}\left(\Ptimat\matr{T}_F^C-\matr{T}_F^C\matr{C}\right).
\end{align*}
Then, with $\matr{T}_F^C = \matr{I}_L\otimes\matr{T}_{F,Q}^C\otimes\matr{T}_{F,A}^C$ we have
\begin{align*}
    \Ptimat\matr{T}_F^C &= \left(\matr{I}_{L\tilde{M}\tilde{N}} - \mu\matr{I}_L\otimes \matr{\tilde{Q}}_\Delta \otimes \matr{\tilde{A}} - \matr{E}\otimes\matr{\tilde{H}}\right)\matr{T}_F^C \\
    &= \matr{T}_F^C\matr{I}_{LMN} - \mu\matr{I}_L\otimes \matr{\tilde{Q}}_\Delta\matr{T}_{F,Q}^C \otimes \matr{\tilde{A}}\matr{T}_{F,A}^C - \matr{T}_F^C\matr{E}\otimes\matr{H}
\end{align*}
and therefore
\begin{align}\label{eq:diff_cgc}
    \matr{T}_{\mathrm{CGC}}(\mu) - \matr{T}_{\mathrm{CGC}}(0) = \mu\matr{T}_C^F\Ptimat^{-1}\left(\matr{I}_L\otimes \left(\matr{T}_{F,Q}^C\matr{Q} \otimes \matr{T}_{F,A}^C\matr{A}-\matr{\tilde{Q}}_\Delta\matr{T}_{F,Q}^C \otimes \matr{\tilde{A}}\matr{T}_{F,A}^C\right)\right).
\end{align}
As before, the norm of the inverse of the coarse-level preconditioner can be bounded by
\begin{align}\label{eq:cgc_mu_bound_to0}
    \left\lVert\Ptimat^{-1}\right\lVert = \left\lVert\Ptimat^{-1}(\mu)\right\lVert = \left\lVert\left(\matr{I}_{L\tilde{M}\tilde{N}}  - \mu\matr{I}_L \otimes \matr{\tilde{Q}_{\Delta}}\otimes\matr{\tilde{A}}- \matr{E} \otimes \matr{\tilde{H}}\right)^{-1}\right\rVert \le c_3(\mu^*_{\mathrm{CGC},0}) = c_3
\end{align}
if $\mu$ is smaller than a given value $\mu^*_{\mathrm{CGC},0}>0$, where the constant $c_3(\mu^*_{\mathrm{CGC},0})$ is again independent of $\mu$.
Together with~\eqref{eq:diff_cgc} this leads to
\begin{align*}
    \left\lVert\matr{T}_{\mathrm{CGC}}(\mu) - \matr{T}_{\mathrm{CGC}}(0)\right\rVert \le c_4\mu
\end{align*}
as for the smoother.
This allows us to write the iteration matrix of the coarse-grid correction as
\begin{align}\label{eq:cgc_perturbed}
    \matr{T}_{\mathrm{CGC}}(\mu) = \matr{T}_{\mathrm{CGC}}(0)+ \mathcal{O}(\mu) = \matr{I}_{LMN} - \matr{T}_C^F\matr{T}_F^C + \mathcal{O}(\mu).
\end{align}
While the eigenvalues of $\matr{T}_{\mathrm{CGC}}(\mu)$ again converge to the eigenvalues of $\matr{T}_{\mathrm{CGC}}(0)$, the eigenvalues of the latter are not zero anymore.
For a partial differential equation in one dimension half of the eigenvalues of $\matr{T}_C^F\matr{T}_F^C $ are zero for standard Lagrangian interpolation and restriction, because the dimension of the coarse space is only of half size.
Therefore, the limit matrix $\matr{T}_{\mathrm{CGC}}(0)$ has a spectral radius of at least $1$.

\subsection{PFASST}

We now couple both results to analyze the full iteration matrix $\matr{T}_{\mathrm{PFASST}} = \matr{T}(\mu)$ of PFASST.
Using~\eqref{eq:smoo_perturbed} and~\eqref{eq:cgc_perturbed}, we obtain
\begin{align*}
    \matr{T}(\mu) &= \matr{T}_{\mathrm{S}}(\mu)\matr{T}_{\mathrm{CGC}}(\mu) = \left(\matr{E} \otimes \matr{H} + \mathcal{O}(\mu)\right)\left(\matr{I}_{LMN} - \matr{T}_C^F\matr{T}_F^C + \mathcal{O}(\mu)\right)\\
    &= \left(\matr{E} \otimes \matr{H}\right)\left(\matr{I}_{LMN} - \matr{T}_C^F\matr{T}_F^C \right) + \mathcal{O}(\mu)\\
    &=\matr{E}\otimes\left(\matr{H}\left(\matr{I}_{MN}-\matr{T}_{C,Q}^F\matr{T}_{F,Q}^C\otimes\matr{T}_{C,A}^F\matr{T}_{F,A}^C\right)\right) + \mathcal{O}(\mu) = \matr{T}(0) + \mathcal{O}(\mu)
\end{align*}
Again, the eigenvalues of $\matr{T}(0)$ are all zero, because the eigenvalues of $\matr{E}$ are all zero.
We can therefore extend Lemma~\ref{lem:smoother_to0} to cover the full iteration matrix of PFASST.

\begin{theorem}\label{th:conv_pfasst_to0}
    PFASST converges for linear problems and Gau\ss-Radau nodes, if the CFL number $\mu$ is small enough.
    More precisely, for $L$ time-steps the spectral radius of the iteration matrix is bounded by
    \begin{align*}
        \rho\left(\matr{T}_{\mathrm{PFASST}}(\mu)\right) \le c\mu^\frac{1}{L}
    \end{align*} 
    for a constant $c>0$ independent of $\mu$, if $\mu < \mu^*_0$ for a fixed value $\mu^*_0>0$.
\end{theorem}

\begin{proof}
    The estimate for the spectral radius can be shown in an analogous way as the one for the smoother in Lemma~\ref{lem:smoother_to0}.
    $\mu^*$ is given by $\mu^*_0 = \min\{\mu^*_{\mathrm{S},0}, \mu^*_{\mathrm{CGC},0}\}$, where $\mu^*_{\mathrm{S},0}$ comes from~\eqref{eq:smoo_mu_bound_to0} and $\mu^*_{\mathrm{CGC},0}$ from~\eqref{eq:cgc_mu_bound_to0}.
\end{proof}

We note that this result indeed proves convergence for PFASST in the non-stiff limit, but it does not provide an indication of the speed of convergence.
It rather ensures convergence for the asymptotic case of $\mu\rightarrow 0$.
To get an estimate of the speed of convergence of PFASST, a bound for the norm of the iteration matrix  would be necessary. 
Yet, we could make the same observation as in Remark~\ref{rem:smoo_norm_bound}, but this is again not a particularly useful result.

\section{The stiff limit}

Working with larger and larger CFL numbers $\mu$ requires an additional trick and poses limitations on the spatial problem. 
Again, we will analyze the iteration matrices of the smoother, the coarse-grid correction and the full PFASST iteration separately.
In the last section we will then discuss the norms of the iteration matrices, which in contrast to the non-stiff limit case leads to interesting new insights.

\subsection{The smoother}\label{sec:smoo_toinf}

In order to calculate the stiff limit $\matr{T}_{\mathrm{S}}(\infty)$ of the iteration matrix of the smoother $\matr{T}_{\mathrm{S}} = \matr{T}_{\mathrm{S}}(\mu)$ we write
\begin{align*}
    \matr{T}_{\mathrm{S}} &= \matr{I}_{LMN} - \Phatmat^{-1}\matr{C}\\
    &= \matr{I}_{LMN} - \left(\matr{I}_{LMN} - \mu \matr{I}_L\otimes\matr{Q}_\Delta \otimes \matr{A} \right)^{-1}\left(\matr{I}_{LMN} - \mu \matr{I}_L\otimes\matr{Q} \otimes \matr{A} - \matr{E}\otimes\matr{H}\right)
\end{align*}
This can be written as
\begin{align*}
    \matr{T}_{\mathrm{S}} &= \matr{I}_{LMN} - \left(\frac{1}{\mu}\matr{I}_{LMN} - \matr{I}_L\otimes\matr{Q}_\Delta \otimes \matr{A} \right)^{-1}\left(\frac{1}{\mu}\matr{I}_{LMN} - \frac{1}{\mu}\matr{E}\otimes\matr{H} - \matr{I}_L\otimes\matr{Q} \otimes \matr{A} \right)
\end{align*}
and for $\mu\rightarrow\infty$ we would expect the limit matrix $\matr{T}_{\mathrm{S}}(\infty)$ to be defined as
\begin{align*}
    \matr{T}_{\mathrm{S}}(\infty) = \matr{I}_L\otimes\left(\matr{I}_M-\matr{Q}_\Delta^{-1}\matr{Q}\right)\otimes\matr{I}_N.
\end{align*}
We need to abbreviate this a bit to see the essential details in the following calculation. 
We define:
\begin{align*}
    \matr{M}_1 &= \mu\matr{I}_L\otimes\matr{Q}_\Delta \otimes \matr{A},\\
    \matr{M}_2 &= \mu\matr{I}_L\otimes\matr{Q} \otimes \matr{A},\\
    \matr{M}_3 &= \matr{I}_{LMN} - \matr{E}\otimes\matr{H},
\end{align*}
so that
\begin{align*}
    \matr{T}_{\mathrm{S}}(\mu)  &= \matr{I} - \left(\matr{I} - \matr{M}_1 \right)^{-1}\left(\matr{M}_3 - \matr{M}_2\right),\\
    \matr{T}_{\mathrm{S}}(\infty) &= \matr{I}  - \left(\matr{M}_1 \right)^{-1}\matr{M}_2,
\end{align*}
omitting the dimension index at the identity matrix.
Then we have
\begin{align*}
    \matr{T}_{\mathrm{S}}(\mu) -  \matr{T}_{\mathrm{S}}(\infty) &= \matr{I} - \left(\matr{I} - \matr{M}_1 \right)^{-1}\left(\matr{M}_3 - \matr{M}_2\right) -  \matr{I}  + \matr{M}_1^{-1}\matr{M}_2\\ 
    &= \matr{M}_1^{-1}\matr{M}_2 + \left(\matr{I} - \matr{M}_1 \right)^{-1}\matr{M}_2 - \left(\matr{I} - \matr{M}_1 \right)^{-1}\matr{M}_3\\
    &= \left(\matr{M}_1^{-1} + \left(\matr{I} - \matr{M}_1 \right)^{-1}\right)\matr{M}_2 - \left(\matr{I} - \matr{M}_1 \right)^{-1}\matr{M}_3\\
    &= \left(\matr{M}_1^{-1}\left(\matr{I} - \matr{M}_1 \right) + \matr{I}\right)\left(\matr{I} - \matr{M}_1 \right)^{-1}\matr{M}_2 - \left(\matr{I} - \matr{M}_1 \right)^{-1}\matr{M}_3\\
    &= \matr{M}_1^{-1}\left(\matr{I} - \matr{M}_1 \right)^{-1}\matr{M}_2 - \left(\matr{I} - \matr{M}_1 \right)^{-1}\matr{M}_3.
\end{align*}
In the first term, the matrix $\matr{M}_1^{-1}$ has a factor $\frac{1}{\mu}$ which is removed by the factor $\mu$ of the matrix $\matr{M}_2$.
In the second term, the matrix $\matr{M}_3$ does not depend on $\mu$, so that in both cases only $\left(\matr{I} - \matr{M}_1 \right)^{-1}$ adds a dependence on $\mu$.
More precisely, following the same argument as for non-stiff part,
\begin{align*}
    \left\lVert\left(\matr{I} - \matr{M}_1 \right)^{-1}\right\rVert &= \left\lVert\left(\matr{I}_{LMN} - \mu\matr{I}_L\otimes\matr{Q}_\Delta \otimes \matr{A}\right)^{-1}\right\rVert\\
    &= \frac{1}{\mu}\left\lVert\left(\frac{1}{\mu}\matr{I}_{LMN} - \matr{I}_L\otimes\matr{Q}_\Delta \otimes \matr{A}\right)^{-1}\right\rVert \le \frac{1}{\mu}c_4(\mu^*_{\mathrm{S}, \infty}) = \frac{1}{\mu}c_4
\end{align*}
for all $\mu\ge \mu^*_{\mathrm{S}, \infty} > 0$, as long as $\left(\matr{Q}_\Delta \otimes \matr{A} \right)^{-1}$ exists.
Then we obtain
\begin{align}\label{eq:smoo_mu_bound_toinf}
    \left\lVert\matr{T}_{\mathrm{S}}(\mu) -  \matr{T}_{\mathrm{S}}(\infty)\right\rVert \le c_5\frac{1}{\mu},
\end{align}
so that the iteration matrix $\matr{T}_{\mathrm{S}}(\mu)$ converges linearly to $\matr{T}_{\mathrm{S}}(\infty)$ as $\mu\rightarrow\infty$ and we can write
\begin{align*}
    \matr{T}_{\mathrm{S}}(\mu) = \matr{T}_{\mathrm{S}}(\infty) + \mathcal{O}\left(\frac{1}{\mu}\right) = \matr{I}_L\otimes\left(\matr{I}_M-\matr{Q}_\Delta^{-1}\matr{Q}\right)\otimes\matr{I}_N + \mathcal{O}\left(\frac{1}{\mu}\right) 
\end{align*}
for $\mu$ large enough.
This result looks indeed very similar to the one obtained for the non-stiff limit, but now the eigenvalues of the limit matrix $\matr{T}_{\mathrm{S}}(\infty)$ are not zero anymore, at least not for arbitrary choices of the preconditioner $\matr{Q}_\Delta$.
Yet, if we choose to define this preconditioner using the LU trick~\cite{Weiser2014}, i.e.~$\matr{Q}_\Delta = \matr{U}^T$ for $\matr{Q}^T = \matr{L}\matr{U}$, we can prove the following analog of Lemma~\ref{lem:smoother_to0}, provided that we couple the number of smoothing steps to the number of collocation nodes.

\begin{lemma}\label{lem:smoo_toinf}
    The smoother of PFASST converges for linear problems and Gau\ss-Radau nodes with preconditioning using LU, if the CFL number $\mu$ is large enough and at least $M$ iterations are performed. We further require that $\matr{A}$ is invertible.
    More precisely, the spectral radius of the iteration matrix $\matr{T}_{\mathrm{S}}\left(\mu,k\right)$ is then bounded by
    \begin{align*}
        \rho\left(\matr{T}_{\mathrm{S}}\left(\mu,k\right)\right) \le c\frac{1}{\mu}
    \end{align*} 
    for a constant $c>0$ independent of $\mu$ but depending on $k$, if $k\ge M$ and $\mu > \mu^*_{\mathrm{S},\infty}$ for a fixed value $\mu^*_{\mathrm{S},\infty} > 0$.
\end{lemma}

\begin{proof}
    For $k\in\mathbb{N}$ iterations of the smoother we have
    \begin{align*}
        \matr{T}_{\mathrm{S}}\left(\mu,k\right) &=  \left(\matr{T}_{\mathrm{S}}(\infty) + \mathcal{O}\left(\frac{1}{\mu}\right)\right)^k = \left(\matr{T}_{\mathrm{S}}(\infty)\right)^k + \mathcal{O}\left(\frac{1}{\mu}\right)
    \end{align*}
    under the conditions of the lemma.
    With $\matr{Q}_\Delta = \matr{U}^T$ and $\matr{Q}^T = \matr{L}\matr{U}$ it is $\matr{I}_M-\matr{Q}_\Delta^{-1}\matr{Q} = \matr{I}_M-\matr{L}^T$, so that $\matr{T}_{\mathrm{S}}(\infty)$ is strictly upper diagonal and therefore nilpotent with $\left(\matr{T}_{\mathrm{S}}(\infty)\right)^M = \matr{0}$.
    Therefore, after at least $M$ iterations the stiff limit matrix of the smoother is actually  $\matr{0}$ and so are its eigenvalues.
    Instead of using the general perturbation result as in the non-stiff limit, we now can apply Elsner's theorem~\cite{stewart_perturbation_1990}, stating that for a perturbation $\matr{D}\in\mathbb{C}^{N\times N}$ of a matrix $\matr{T}\in\mathbb{C}^{N\times N}$, i.e. for $\matr{\hat{T}} = \matr{T} + \matr{D}$ it is
    \begin{align}
        \max_{i=1,...,N}\min_{j=1,...,N}\lvert\lambda_i(\matr{\hat{T}})-\lambda_j(\matr{T})\rvert \le \left(\lVert\matr{\hat{T}}\rVert_2 + \lVert\matr{T}\rVert_2\right)^{1-1/N}\lVert\matr{D}\rVert_2^{1/N},
        \label{eq:pertrub_ineq}
    \end{align}
    where $\lambda_i(\matr{\hat{T}})$ corresponds to the $i$th eigenvalue of $\matr{\hat{T}}$ and $\lambda_j(\matr{T})$ to the $j$th eigenvalue of $\matr{T}$.
    In our case, $\matr{T} = \left(\matr{T}_{\mathrm{S}}(\infty)\right)^k = \matr{0}$ for $k \ge M$, so that the left-hand side of this inequality is just the spectral radius of $\matr{\hat{T}} = \matr{T}_{\mathrm{S}}\left(\mu,k\right)$.
    The norm of the perturbation matrix $\matr{D}$ is bounded via~\eqref{eq:smoo_mu_bound_toinf}, so that for $L$ steps, $M$ nodes, $N$ degrees-of-freedom, $k\ge M$ iterations we have
    \begin{align*}
         \rho\left(\matr{T}_{\mathrm{S}}\left(\mu,k\right)\right)  &\le \left(\left\lVert\matr{T}_{\mathrm{S}}\left(\mu,k\right)\right\rVert_2 + \left\lVert\matr{T}_{\mathrm{S}}\left(\infty,k\right)\right\rVert_2\right)^{1-\frac{1}{LMN}}\left\lVert\matr{D}\right\rVert_2^\frac{1}{LMN}\\
        &\le \left(2\left\lVert\matr{0}\right\rVert_2 + \left\lVert\matr{D}\right\rVert_2\right)^{1-\frac{1}{LMN}}\left\lVert\matr{D}\right\rVert_2^\frac{1}{LMN} =  \left\lVert\matr{D}\right\rVert_2\\
        &= \left\lVert\matr{T}_{\mathrm{S}}(\mu, k) -  \matr{T}_{\mathrm{S}}(\infty, k)\right\rVert \le c\frac{1}{\mu}.
    \end{align*}
    by using $\left\lVert\matr{T}_{\mathrm{S}}(\mu, k)\right\rVert \le \left\lVert\matr{T}_{\mathrm{S}}(\infty, k)\right\rVert + \left\lVert\matr{D}\right\rVert$ for the second inequality.
\end{proof}

\begin{figure}[t!]
  \centering
  \begin{subfigure}[b]{0.475\textwidth}
    \centering
    \includegraphics[width=\textwidth]{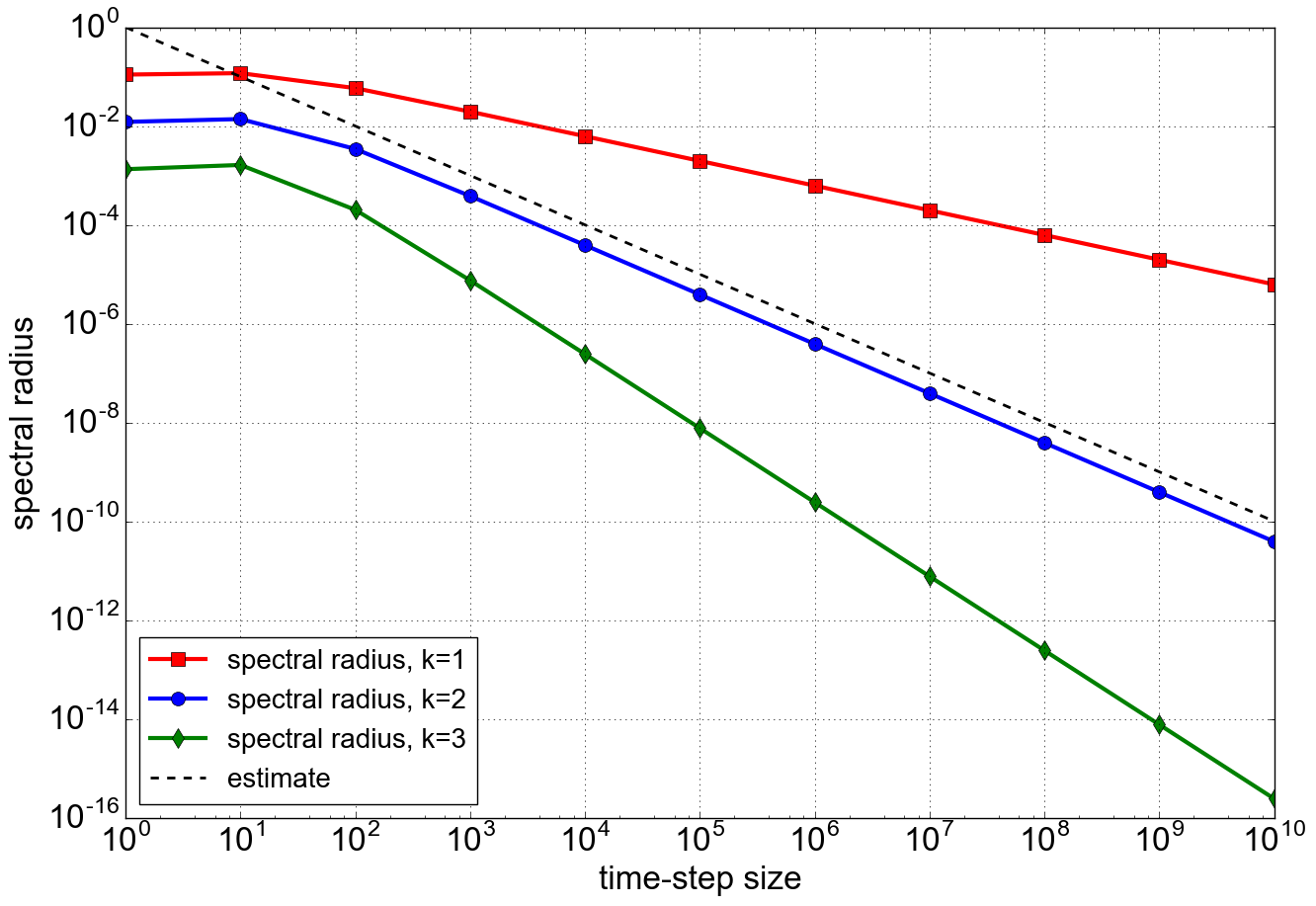}
    \caption{$M=3$ and $\lambda=-1$}
    \label{fig:smoother_specrad_toinf_M3_LU_real}
  \end{subfigure}
  \begin{subfigure}[b]{0.475\textwidth}
    \centering
    \includegraphics[width=\textwidth]{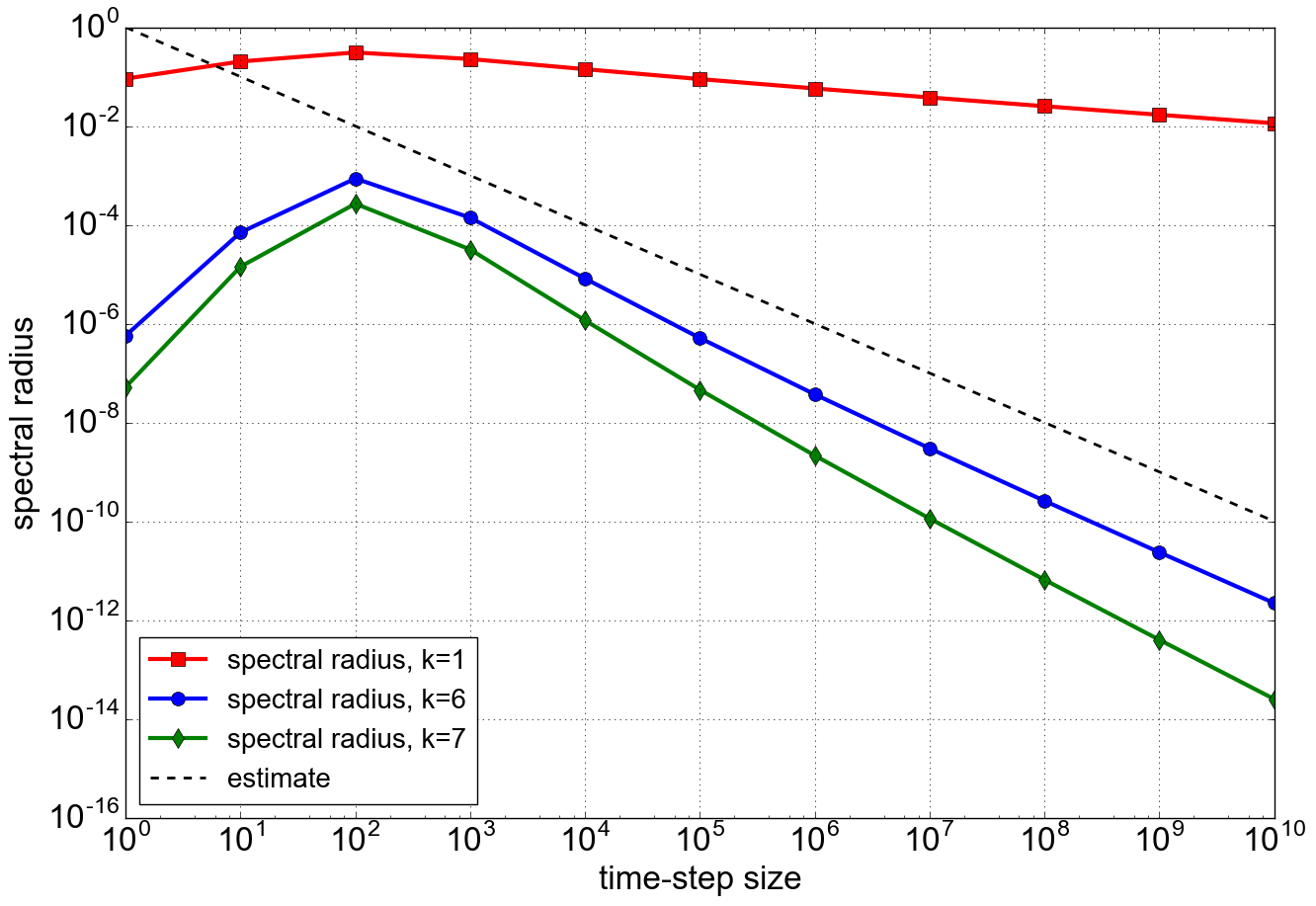}
    \caption{$M=7$ and $\lambda=-1$}
    \label{fig:smoother_specrad_toinf_M7_LU_real}
  \end{subfigure}
  \begin{subfigure}[b]{0.475\textwidth}
    \centering
    \includegraphics[width=\textwidth]{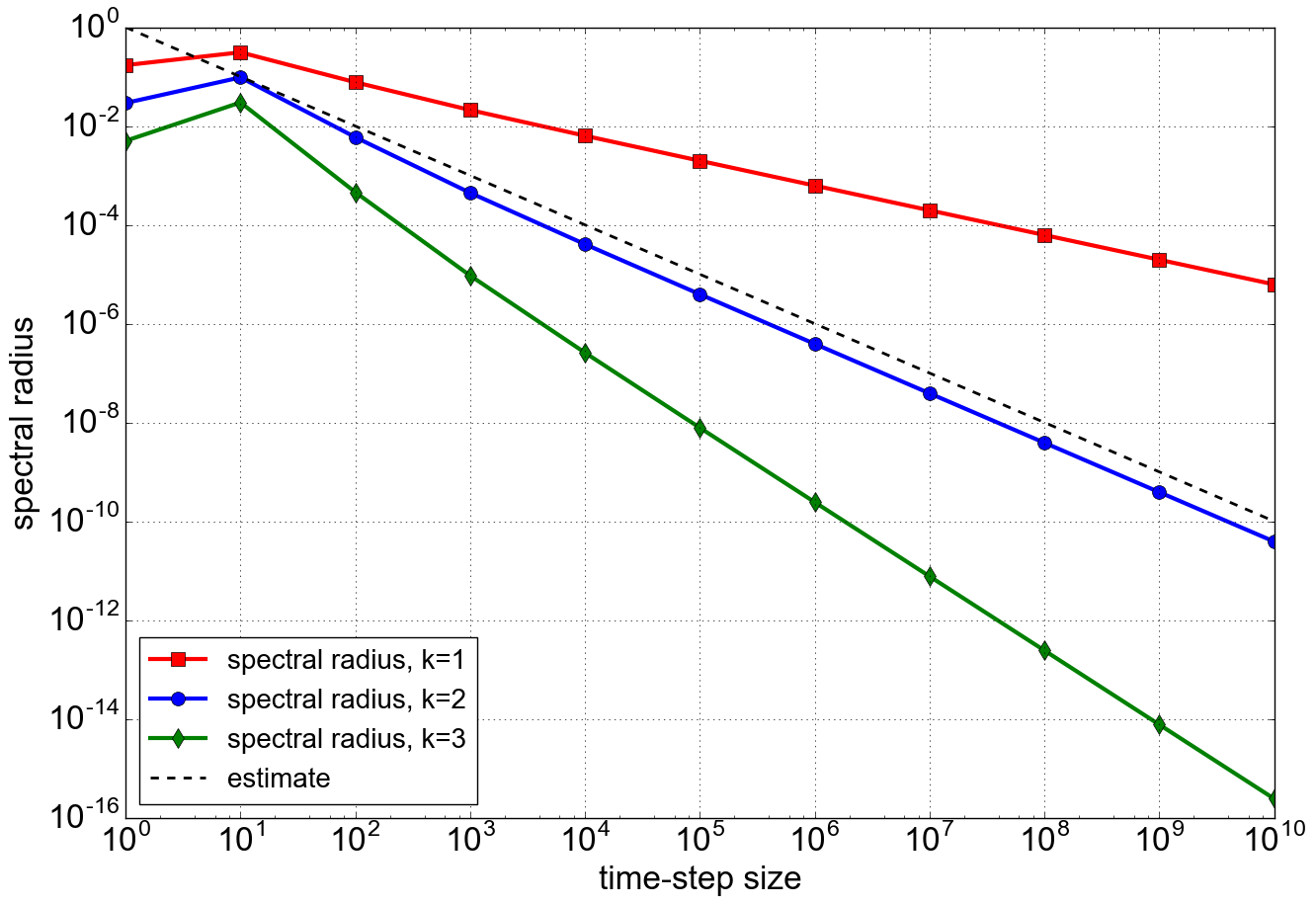}
    \caption{$M=3$ and $\lambda=i$}
    \label{fig:smoother_specrad_toinf_M3_LU_img}
  \end{subfigure}
  \begin{subfigure}[b]{0.475\textwidth}
    \centering
    \includegraphics[width=\textwidth]{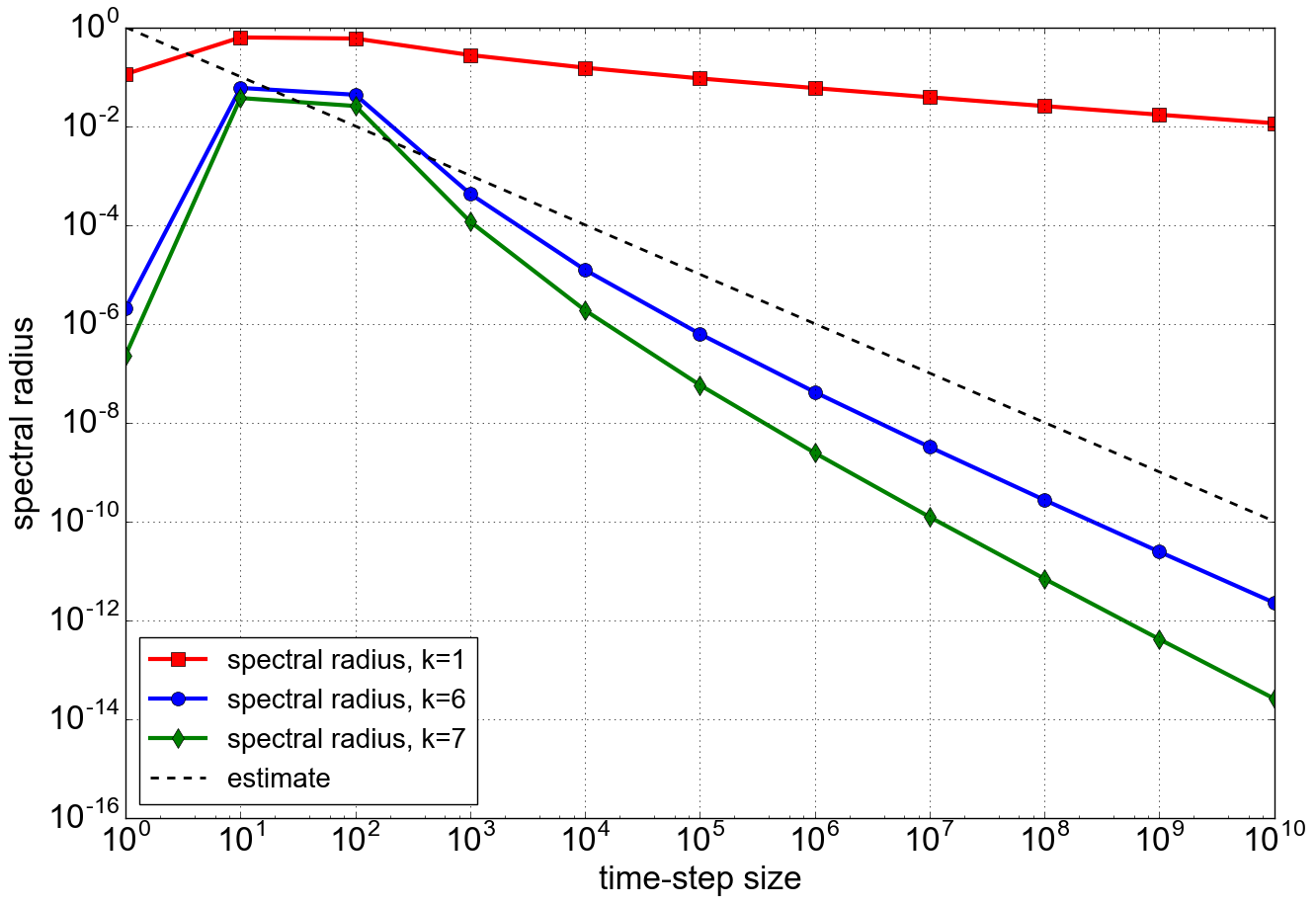}
    \caption{$M=7$ and $\lambda=i$}
    \label{fig:smoother_specrad_toinf_M7_LU_img}
  \end{subfigure}  
  \caption{Spectral radius of the iteration matrix of the smoother with LU decomposition for $\lambda = -1$ (upper) and $\lambda=i$ (lower), $L=1$. Left: $M=3$, right: $M=7$.}
  \label{fig:smoother_specrad_toinf}
\end{figure}

This result shows that the spectral radius goes to zero with the same speed as $\mu$ goes to infinity, independently of the number of time-step, the number of collocation nodes or the spatial resolution.
Yet, the condition of requiring at least $M$ smoothing steps is rather severe since standard simulations with PFASST typically perform only one or two steps here.
We show in Figure~\ref{fig:smoother_specrad_toinf} the spectral radii of the iteration matrix of the smoother for Dahlquist's test problem $u_t = \lambda u$, $u(0) = 1$ with $\lambda = -1$, $\lambda=i$ (upper and lower figures) and larger and larger time-step sizes $\dt$ (with $\mu = \lambda\dt$), using $M=3$ and $M=7$ Gau\ss-Radau nodes (left and right figures).
There are a few interesting things to note here: 
First, for the asymptotics it does not matter which $\lambda\in\mathbb{C}$ of the two we choose, the plots look the same.
Only for small values of $\dt$ the spectral radii differ significantly, leading to a severely worse convergence behavior for complex eigenvalues.
Second, the result does not depend on the number of time-steps $L$, so choosing $L=1$ is reasonable here (and the plots do not change for $L>1$, which is a key difference to the results for $\mu\rightarrow0$).
Third, the estimate of the spectral radius is rather pessimistic, again.
More precisely, already for $k=M-1$ the linear slope of the estimation is met (and, actually, surpassed), while for $k>M-1$ the convergence is actually faster.
This is not reflected in the result above, but it shows that this estimate is only a rather loose bound.
However, we note that in Section~\ref{sec:summary} we see a different outcome, more in line with the theoretical estimate.
Note that in Figure~\ref{fig:smoother_specrad_toinf}, the constant $c$ is not included so that the estimate depicted there is not necessarily an upper bound, but reflects the asymptotics. 

\bigskip

For a more complete overview of the smoother, Figure~\ref{fig:smoother_specrad_full} shows the spectral radius of its iteration matrix for fixed $L$, $M$ and Dahlquist's test for different values $\mu = \dt\lambda$, choosing the LU trick (Fig.~\ref{fig:heatmap_smoother_full_Nsteps4_M3_LU}) and the classical implicit Euler method or right-hand side rule (Fig.~\ref{fig:heatmap_smoother_full_Nsteps4_M3_IE}) for $\matr{Q}_\Delta$.
Following Theorem~\ref{lem:smoo_toinf}, we choose $k=M=3$ smoothing steps.
Note that the scale of the colors is logarithmic, too, and all values for the spectral radius are considerably below $1$.
The largest for the LU trick is at about $0.05$ and for the implicit Euler at about $0.17$.
We can nicely see how the smoother converges for $\mu\rightarrow 0$ as well as for $\mu\rightarrow\infty$ in both cases, although the LU case is much more pronounced.
However, we also see that there is a banded area, where the spectral radius is significantly higher than in the regions of large or small $\mu$.
This is most likely due to the properties of the underlying preconditioner and can be observed on the real axis for standard SDC as well, see~\cite{Weiser2014}.
Oddly, this band does not describe a circle but rather a box, hitting the axes with its maximum at around $-10$ and $10i$.
It seems unlikely that a closed description of this area, i.e.\ filling the gap between $\mu\rightarrow0$ and $\mu\rightarrow\infty$ is easily possible.

\begin{figure}[t!]
  \centering
  \begin{subfigure}[b]{0.475\textwidth}
    \centering
    \includegraphics[width=\textwidth]{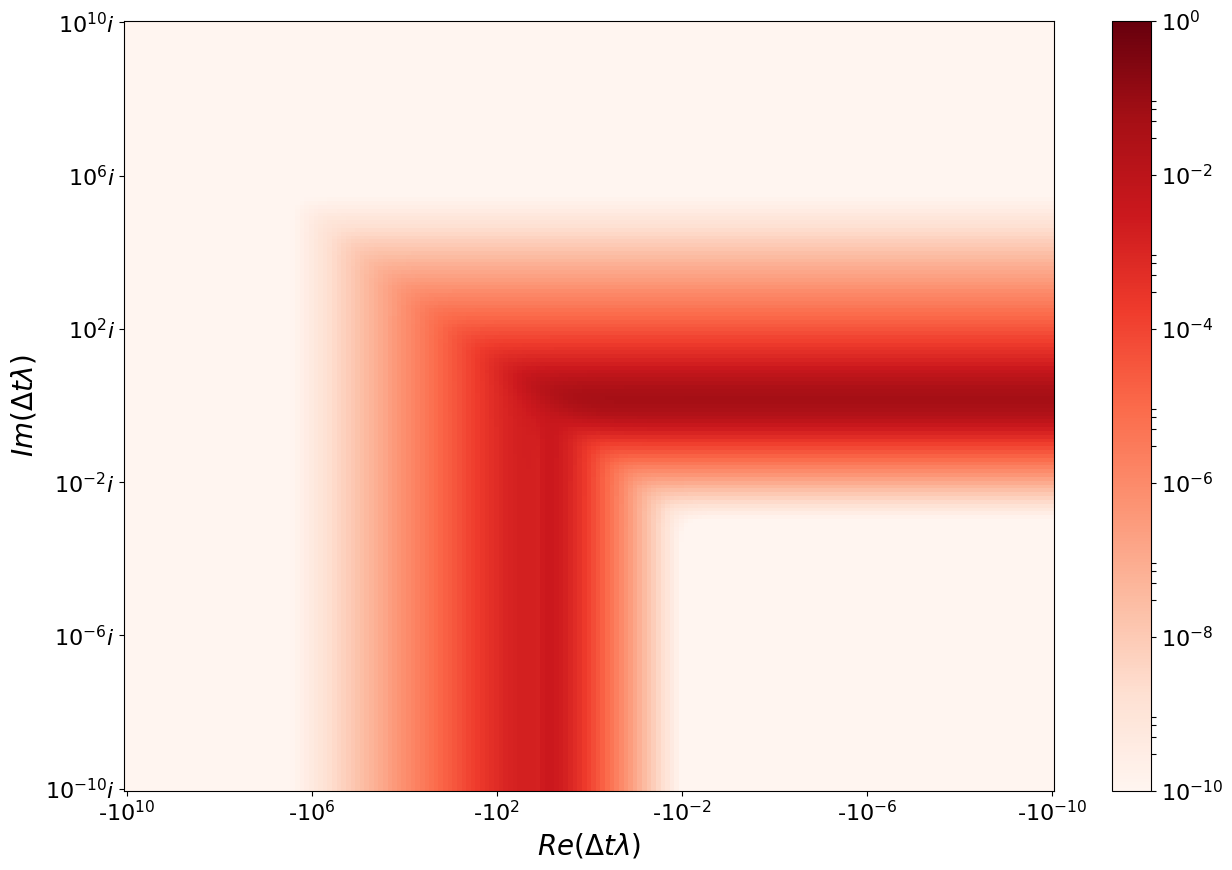}
    \caption{LU-trick}
    \label{fig:heatmap_smoother_full_Nsteps4_M3_LU}
  \end{subfigure}
  \begin{subfigure}[b]{0.475\textwidth}
    \centering
    \includegraphics[width=\textwidth]{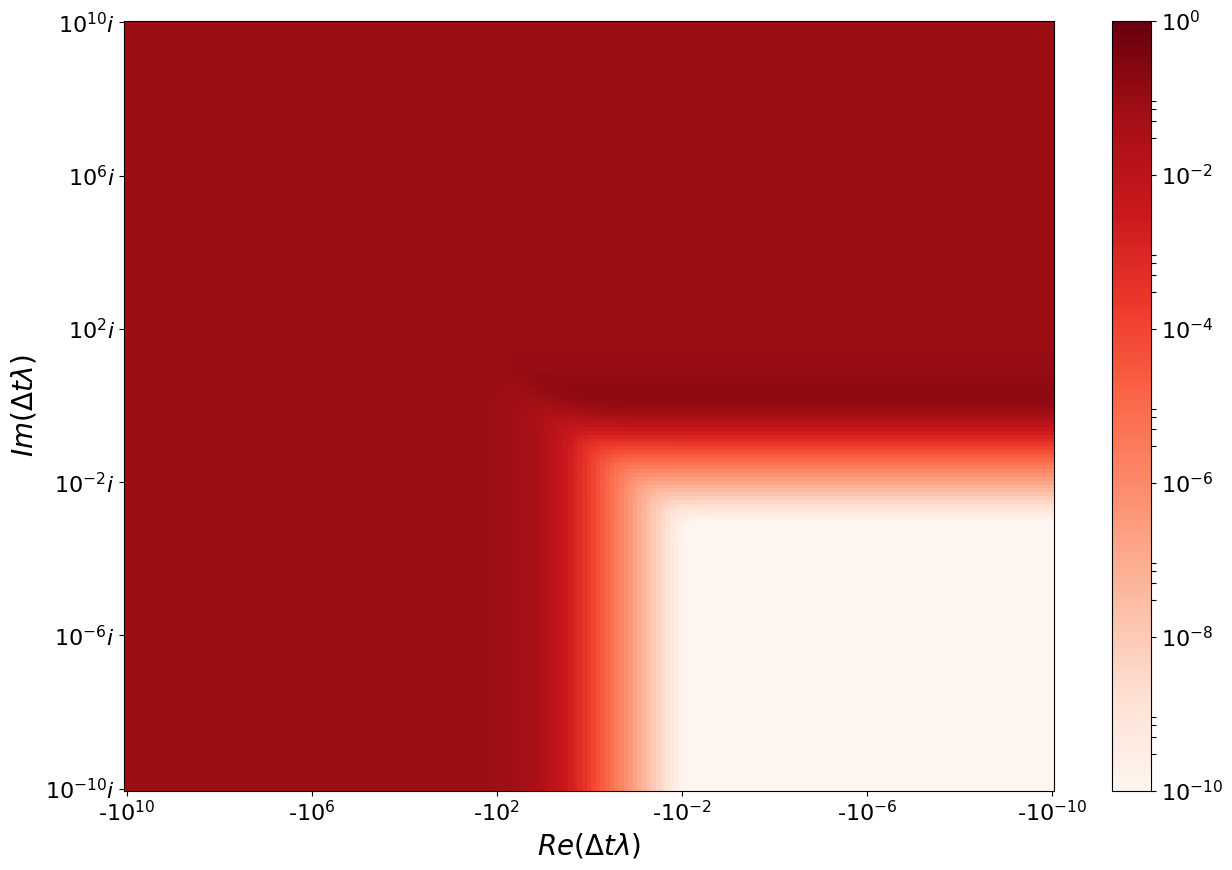}
    \caption{Implicit Euler}
    \label{fig:heatmap_smoother_full_Nsteps4_M3_IE}
  \end{subfigure}    
  \caption{Spectral radius of the iteration matrix of the smoother for different complex values of $\dt\lambda$. Fixed values of $M=3$ and $L=4$, $k=M=3$ smoothing steps. Left: LU-trick for $\matr{Q}_\Delta$, right: implicit Euler for $\matr{Q}_\Delta$.}
  \label{fig:smoother_specrad_full}
\end{figure}

\subsection{The coarse-grid correction}\label{sec:cgc_toinf}

We now focus on the iteration matrix $\matr{T}_{\mathrm{CGC}}= \matr{T}_{\mathrm{CGC}}(\mu)$ of the coarse-grid correction with
\begin{align*}
  \matr{T}_{\mathrm{CGC}} = \matr{I}_{LMN} - \matr{T}_C^F \Ptimat^{-1} \matr{T}_F^C \matr{C}
  = \matr{I}_{LMN} - &\matr{T}_C^F\left(\matr{I}_{L\tilde{M}\tilde{N}} - \mu\matr{I}_L\otimes \matr{\tilde{Q}}_\Delta \otimes \matr{\tilde{A}} - \matr{E}\otimes\matr{\tilde{H}}\right)^{-1}\cdot\\
  &\matr{T}_F^C\left(\matr{I}_{LMN} - \mu \matr{I}_L\otimes\matr{Q} \otimes \matr{A} - \matr{E}\otimes\matr{H}\right).
\end{align*}
Using abbreviations
\begin{align*}
    \matr{\tilde{M}}_1 &= \mu\matr{I}_L\otimes \matr{\tilde{Q}}_\Delta \otimes \matr{\tilde{A}},\\
    \matr{M}_2 &= \mu\matr{I}_L\otimes\matr{Q} \otimes \matr{A},\\
    \matr{M}_3 &= \matr{I}_{LMN} - \matr{E}\otimes\matr{H},\\
    \matr{\tilde{M}}_3 &=\matr{I}_{L\tilde{M}\tilde{N}} - \matr{E}\otimes\matr{\tilde{H}}
\end{align*}
we have
\begin{align*}
    \matr{T}_{\mathrm{CGC}}(\mu)  = \matr{I} - \matr{T}_C^F\left(\matr{\tilde{M}}_3 - \matr{\tilde{M}}_1\right)^{-1}\matr{T}_F^C\left(\matr{M}_3 - \matr{M}_2\right).
\end{align*}
If $\left(\matr{\tilde{Q}}_\Delta \otimes \matr{\tilde{A}} \right)^{-1}$ exists, we define
\begin{align*}
    \matr{T}_{\mathrm{CGC}}(\infty) = \matr{I}_{LMN} - \matr{T}_C^F \left(\matr{I}_L\otimes\matr{\tilde{Q}}_\Delta \otimes \matr{\tilde{A}}\right)^{-1} \matr{T}_F^C \left(\matr{I}_L\otimes\matr{Q} \otimes \matr{A}\right)
\end{align*}
and write it more briefly as
\begin{align*}
    \matr{T}_{\mathrm{CGC}}(\infty)  = \matr{I} - \matr{T}_C^F\matr{\tilde{M}}_1^{-1}\matr{T}_F^C\matr{M}_2.
\end{align*}
With the same ideas and rearrangements we have used for the smoother in the previous section, we obtain after somewhat lengthy algebra
\begin{align*}
    \matr{T}_{\mathrm{CGC}}(\mu)  - \matr{T}_{\mathrm{CGC}}(\infty) = \matr{T}_C^F\matr{\tilde{M}}_1^{-1}\matr{\tilde{M}}_3\left(\matr{\tilde{M}}_3 - \matr{\tilde{M}}_1\right)^{-1}\matr{T}_F^C\matr{M}_2 - \matr{T}_C^F\left(\matr{\tilde{M}}_3 - \matr{\tilde{M}}_1\right)^{-1}\matr{T}_F^C\matr{M}_3.
\end{align*}
In the first term, $\matr{\tilde{M}}_1^{-1}$ adds a factor $\frac{1}{\mu}$, which is removed by $\mu$ from $\matr{M}_2$. $\matr{M}_3$ does not have a dependence on $\mu$, so for both terms only $\left(\matr{\tilde{M}}_3 - \matr{\tilde{M}}_1\right)^{-1} $ is relevant.
We have
\begin{align}
    \begin{split}\label{eq:est_M3M1}
        \left\lVert\left(\matr{\tilde{M}}_3 - \matr{\tilde{M}}_1\right)^{-1}\right\rVert &= \left\lVert\left(\matr{I}_{L\tilde{M}\tilde{N}} - \mu\matr{I}_L\otimes \matr{\tilde{Q}}_\Delta \otimes \matr{\tilde{A}} - \matr{E}\otimes\matr{\tilde{H}}\right)^{-1}\right\rVert\\
        &= \frac{1}{\mu}\left\lVert\left(\frac{1}{\mu}\matr{I}_{L\tilde{M}\tilde{N}} - \matr{I}_L\otimes \matr{\tilde{Q}}_\Delta \otimes \matr{\tilde{A}} - \frac{1}{\mu}\matr{E}\otimes\matr{\tilde{H}}\right)^{-1}\right\rVert\\
        &\le \frac{1}{\mu}c_6(\mu^*_{\mathrm{CGC}, \infty}) = \frac{1}{\mu}c_6
    \end{split}
\end{align}
for all $\mu\ge \mu^*_{\mathrm{CGC}, \infty} > 0$ with the same argument as before.
Therefore, $\matr{T}_{\mathrm{CGC}}(\mu)$ converges to $\matr{T}_{\mathrm{CGC}}(\infty)$ linearly as $\mu\rightarrow\infty$ and we can write
\begin{align*}
     \matr{T}_{\mathrm{CGC}}\left(\mu\right) &= \matr{T}_{\mathrm{CGC}}(\infty) + \mathcal{O}\left(\frac{1}{\mu}\right)\\
     &= \matr{I}_{LMN} - \matr{T}_C^F \left(\matr{I}_L\otimes\matr{\tilde{Q}}_\Delta \otimes \matr{\tilde{A}}\right)^{-1} \matr{T}_F^C \left(\matr{I}_L\otimes\matr{Q} \otimes \matr{A}\right) + \mathcal{O}\left(\frac{1}{\mu}\right),
 \end{align*}
for $\mu$ large enough.

\subsection{PFASST}

We can now combine both parts to obtain an analog of Theorem~\ref{th:conv_pfasst_to0}.

\begin{theorem}\label{th:conv_pfasst_muinf}
    PFASST converges for linear problems and Gau\ss-Radau nodes with preconditioning using LU, if the CFL number $\mu$ is large enough and at least $M$ smoothing steps are performed. We further require that $\matr{A}$ and $\matr{\tilde{A}}$ are invertible.
    More precisely, the spectral radius of the iteration matrix $\matr{T}_{\mathrm{PFASST}}\left(\mu,k\right)$ is then bounded by
    \begin{align*}
        \rho\left(\matr{T}_{\mathrm{PFASST}}\left(\mu,k\right)\right) \le c\frac{1}{\mu}
    \end{align*} 
    for a constant $c>0$ independent of $\mu$ but depending on $k$, if $k\ge M$ and $\mu > \mu^*_\infty$ for a fixed value $\mu^*_\infty > 0$.
\end{theorem}

\begin{proof}
    For $k\in\mathbb{N}$ iterations of the smoother we have with Lemma~\ref{lem:smoo_toinf} and the derivations from the previous section
    \begin{align*}
        \matr{T}_{\mathrm{PFASST}}\left(\mu,k\right) &= \left(\matr{T}_{\mathrm{S}}\left(\mu\right)\right)^k\matr{T}_{\mathrm{CGC}}\left(\mu\right)
        = \left(\matr{T}_{\mathrm{S}}(\infty)\right)^k\matr{T}_{\mathrm{CGC}}(\infty) + \mathcal{O}\left(\frac{1}{\mu}\right) = \mathcal{O}\left(\frac{1}{\mu}\right).
    \end{align*}
    The estimate of the spectral radius uses again Elsner's theorem from~\cite{stewart_perturbation_1990}, see the proof of Lemma~\ref{lem:smoo_toinf}.
\end{proof}

\begin{remark}
  Since the inverses of $\matr{Q}_\Delta$, $\matr{\tilde{Q}}_\Delta$, $\matr{A}$ and $\matr{\tilde{A}}$ appear in the limit matrices, their existence is essential and the assumptions of this lemma are a common theme in this section.
  We need to point out, however, that assuming the existence of these inverses is rather restrictive, since problems like the 1D heat equation on periodic boundaries are already excluded.
  However, this assumption can be weakened by considering a pseudoinverse that acts on the orthogonal complement of the null space of $\matr{A}$ or $\matr{\tilde{A}}$, only.
  In the case of e.g.\ periodic boundary conditions this results in considering the whole space except for vectors representing a constant function, see also the discussion in Section~\ref{sec:summary}.
\end{remark}

\subsection{Smoothing and approximation property}

For the stiff limit case, bounding the norms of the iteration matrices actually provides insight into the relationship between PFASST and standard multigrid theory. 
Following~\cite{hackbusch1985multi} we now analyze the smoothing and the approximation property of PFASST.
Note that we now interpret $\mu\rightarrow\infty$ as $\dx\rightarrow 0$ so that in the following results the appearance of $\mu$ in the denominator and numerator is counterintuitive.
We start with the approximation property, which is straightforward to show.

\begin{lemma}\label{lem:approx_prop_muinf}
    The coarse-grid correction of PFASST for linear problems satisfies the approximation property if the CFL number $\mu$ is large enough and if $\matr{A}$ and $\matr{\tilde{A}}$ are invertible, i.e.~it holds
    \begin{align*}
        \left\lVert\matr{C}^{-1} - \matr{T}_C^F \Ptimat^{-1} \matr{T}_F^C\right\rVert \le c\frac{1}{\mu}
    \end{align*}
    for $\mu$ large enough, with a constant $c>0$ independent of $\mu$.
\end{lemma}

\begin{proof}
    As in Section~\ref{sec:cgc_toinf} we make use of the abbreviations
    \begin{align*}
        \matr{\tilde{M}}_1 &= \mu\matr{I}_L\otimes \matr{\tilde{Q}}_\Delta \otimes \matr{\tilde{A}},\\
        \matr{M}_2 &= \mu\matr{I}_L\otimes\matr{Q} \otimes \matr{A},\\
        \matr{M}_3 &= \matr{I}_{LMN} - \matr{E}\otimes\matr{H},\\
        \matr{\tilde{M}}_3 &=\matr{I}_{L\tilde{M}\tilde{N}} - \matr{E}\otimes\matr{\tilde{H}}
    \end{align*}
    and write
    \begin{align*}
        \matr{C}^{-1} - \matr{T}_C^F \Ptimat^{-1} \matr{T}_F^C = \left(\matr{M}_3 - \matr{M}_2\right)^{-1} - \matr{T}_C^F \left(\matr{\tilde{M}}_3 - \matr{\tilde{M}}_1\right)^{-1} \matr{T}_F^C.
    \end{align*}
    With~\eqref{eq:est_M3M1} the second term can be bounded by
    \begin{align*}
        \left\lVert\matr{T}_C^F \left(\matr{\tilde{M}}_3 - \matr{\tilde{M}}_1\right)^{-1} \matr{T}_F^C\right\rVert \le c_6\frac{1}{\mu}
    \end{align*}
    if $\mu$ is large enough.
    In the very same way we can also bound the first term, i.e.
    \begin{align*}
        \left\lVert\left(\matr{M}_3 - \matr{M}_2\right)^{-1} \matr{T}_F^C\right\rVert \le c_7\frac{1}{\mu},
    \end{align*}
    for $\mu$ large enough, so that 
    \begin{align*}
        \left\lVert\matr{C}^{-1} - \matr{T}_C^F \Ptimat^{-1} \matr{T}_F^C \matr{T}_F^C\right\rVert \le c_6\frac{1}{\mu} + c_7\frac{1}{\mu} \le c\frac{1}{\mu}.
    \end{align*}
\end{proof}

A natural question to ask is whether the smoother satisfies the smoothing property, which would make PFASST an actual multigrid algorithm in the classical sense.
However, this is not the case, as we can also observe numerically~\cite{doi:10.1002/nla.2110}.
Still, we can bound the norm of the iteration matrix.

\begin{lemma}\label{lem:smoo_prop_muinf}
    For the CFL number $\mu$ large enough and $k\in\mathbb{N}$ iterations, we have
    \begin{align*}
        \left\lVert\matr{C}\left(\matr{I}_{LMN} - \Phatmat^{-1}\matr{C}\right)^k\right\rVert \le \mu\left(c\left\lVert\matr{I}_L\otimes\left(\matr{I}_M-\matr{Q}_\Delta^{-1}\matr{Q}\right)^k\otimes\matr{I}_N\right\rVert + \mathcal{O}\left(\frac{1}{\mu}\right)\right),
    \end{align*}
    if $\matr{A}^{-1}$ exists and the LU trick is used for $\matr{Q}_\Delta$.
\end{lemma}

\begin{proof}
    Showing this is rather straightforward using the derivation of Lemma~\ref{lem:smoo_toinf} and by realizing that
    \begin{align*}
        \matr{C} = \mu\left(\frac{1}{\mu}\matr{I}_{LMN} - \matr{I}_L\otimes\matr{Q} \otimes \matr{A} - \frac{1}{\mu}\matr{E}\otimes\matr{H}\right)
    \end{align*}
    and therefore $\left\lVert\matr{C}\right\rVert  \le c_8\mu$ for $\mu$ large enough.
\end{proof}

At first glance this does look like the standard smoothing property for multigrid methods after all, where we would expect
\begin{align*}
    \left\lVert\matr{C}\left(\matr{T}_{\mathrm{S}}\left(\mu\right)\right)^k\right\rVert  = \left\lVert\matr{C}\left(\matr{I}_{LMN} - \Phatmat^{-1}\matr{C}\right)^k\right\rVert \le \mu g(k)
\end{align*}
with $g(k)\rightarrow 0$ for $k\rightarrow\infty$ independent of $\mu$.
In our case, however, $g(k) = \mathcal{O}\left(\frac{1}{\mu}\right)$ if $k$ is larger than the number of quadrature nodes $M$, see Lemma~\ref{lem:smoo_toinf}.
Even worse, this implies that 
\begin{align*}
    \left\lVert\matr{C}\left(\matr{I}_{LMN} - \Phatmat^{-1}\matr{C}\right)^k\right\rVert \le c
\end{align*}
for some constant $c>0$, if $k\ge M$, so that in this case the norm of the smoother does not even converge to zero, if $\mu$ goes to infinity. 
It is not even guaranteed that this constant $c$ is below 1.

\bigskip

However, we can couple Lemmas~\ref{lem:approx_prop_muinf} and~\ref{lem:smoo_prop_muinf} to bound the norm of the full iteration matrix of PFASST.

\begin{theorem}\label{th:norm_bound_PFASST}
    If Lemmas~\ref{lem:approx_prop_muinf} and~\ref{lem:smoo_prop_muinf} hold true, then the norm of the iteration matrix $\matr{T}_{\mathrm{PFASST}}\left(\mu,k\right)$ of PFASST with $k$ pre-smoothing steps can be bounded by
    \begin{align*}
        \left\lVert\matr{T}_{\mathrm{PFASST}}\left(\mu,k\right)\right\rVert \le c\left\lVert\matr{I}_L\otimes\left(\matr{I}_M-\matr{Q}_\Delta^{-1}\matr{Q}\right)^k\otimes\matr{I}_N\right\rVert + \mathcal{O}\left(\frac{1}{\mu}\right)
    \end{align*}
    so that for $k\ge M$ the norm of the iteration matrix goes to zero as $\mu\rightarrow \infty$.
\end{theorem}

\begin{proof}
    The proof is straightforward, but we note that we used pre-smoothing here instead of post-smoothing as in Theorem~\ref{th:pfasst_in_matrix_form}, which in the norm does not matter.
\end{proof}

We see that this gives nearly $\mu$-independent convergence of PFASST as for classical multigrid methods.
The last term in the bound becomes smaller and smaller when $\mu$ becomes larger and larger so that at least asymptotically convergence speed is increased for $\mu\rightarrow\infty$.

\bigskip

Thus, the only fundamental difference between PFASST and classical linear multigrid methods lies in the smoothing property.
For the standard approximative block Jacobi method, this does not seem to hold.
Yet, there is another approach which helps us here, namely Reusken's Lemma~\cite{reusken1991new,braess2001finite}.
To apply this, we define the modified approximative block Jacobi preconditioner by
\begin{align*}
  \Phatmat_\omega = \matr{I}_{LMN} - \mu \matr{I}_L\otimes\omega\matr{Q}_\Delta \otimes \matr{A} 
\end{align*}
and the corresponding iteration matrix $\matr{T}_\omega$ by
\begin{align*}
   \matr{T}_\omega = \matr{I}_{LMN} - \Phatmat^{-1}_\omega\matr{C}
\end{align*}
for some parameter $\omega>0$.

\begin{lemma}\label{lem:smoo_prop_muinf_mod}
  We assume that the number of quadrature nodes $M$ is small enough, $\matr{Q}_\Delta$ is given by the LU trick, $\matr{A}$ is invertible and $\mu$ is large enough. Then the iteration matrix $\matr{T}_2$ satisfies the smoothing property. More precisely, we have
  \begin{align*}
    \left\lVert\matr{C}\left(\matr{I}_{LMN} - \Phatmat^{-1}_2\matr{C}\right)^k\right\rVert \le c\sqrt{\frac{8}{k\pi}}\mu
  \end{align*}
  for $k$ smoothing steps, $\mu$ large enough and $M\le M^*$, where $M^*$ depends on the matrix norm.
\end{lemma}

\begin{proof}
  The basis for this proof as well as for the choice of the parameter $\omega$ is Reusken's Lemma, stating that for some invertible matrix $\matr{P}$ and an iteration matrix $\matr{T} = \frac{1}{2}(\matr{I} + \matr{B})$ with $\matr{B} = \matr{I} - \matr{P}^{-1}\matr{C}$ it is
  \begin{align*}
    \left\lVert\matr{C}\matr{T}^k\right\rVert \le \sqrt{\frac{8}{k\pi}}\left\lVert\matr{P}\right\rVert
  \end{align*}
  if $\left\lVert\matr{B}\right\rVert \le 1$.
  In our case we have
  \begin{align*}
    \matr{B} = \matr{B}_\omega = \matr{I} - 2\Phatmat^{-1}_\omega\matr{C}\quad \text{and}\quad \matr{P} = \matr{P}_\omega = \frac{1}{2}\Phatmat_\omega
  \end{align*}
  and we write
  \begin{align*}
    \matr{B}_\omega = \matr{B}_\omega\left(\mu\right) = \matr{B}_\omega\left(\infty\right) + \mathcal{O}\left(\frac{1}{\mu}\right) = \matr{I}_L\otimes\left(\matr{I}_M-\frac{2}{\omega}\matr{Q}_\Delta^{-1}\matr{Q}\right)\otimes\matr{I}_N + \mathcal{O}\left(\frac{1}{\mu}\right),
  \end{align*}
  with $\matr{B}_\omega(\infty)$ analogously derived as $\matr{T}_{\mathrm{S}}(\infty)$ in Section~\ref{sec:smoo_toinf}.
  Now, bounding the norm of $\matr{B}_\omega(\infty)$ is by far not straightforward and we fall back on numerical calculations to find scenarios where $\left\lVert\matr{B}_\omega(\infty)\right\rVert <1$.
  In particular, if we choose $\omega = 2$, $\left\lVert.\right\rVert$ as the infinity-norm and $M\le5$, then $\left\lVert\matr{B}_\omega(\infty)\right\rVert \approx 0.8676$.
  In the 2-norm, $M=6,7$ is also allowed and choosing $\omega>2$ extends this range further.
  Anyway, we are able to bound $\left\lVert\matr{B}_2(\infty)\right\rVert$ and thus $\left\lVert\matr{B}_2\right\rVert$ by one, provided we have chosen the parameters carefully.
  Then, Reusken's Lemma is applicable and we have
  \begin{align*}
    \left\lVert\matr{C}\left(\matr{I}_{LMN} - \Phatmat^{-1}_2\matr{C}\right)^k\right\rVert \le \sqrt{\frac{8}{k\pi}}\left\lVert\matr{P}_2\right\rVert.
  \end{align*}
  Now, the norm of $\matr{P}_2$ can be bounded by 
  \begin{align*}
    \left\lVert\matr{P}_2\right\rVert &= \left\lVert\frac{1}{2}\left(\matr{I}_{LMN} - \mu \matr{I}_L\otimes2\matr{Q}_\Delta \otimes \matr{A}\right)\right\rVert
    \le \mu\left\lVert\frac{1}{2\mu}\matr{I}_{LMN} - \matr{I}_L\otimes\matr{Q}_\Delta \otimes \matr{A}\right\rVert \le c\mu
  \end{align*}
  if $\mu\ge\mu^*$ for some $\mu^*$, which concludes the proof.
\end{proof}

Although the assumptions in Lemma~\ref{lem:smoo_prop_muinf_mod} are more restrictive that those of Lemma~\ref{lem:smoo_prop_muinf}, we now have an algorithm satifying both smoothing and approximation property, i.e.~a multigrid algorithm with $\mu$-independent convergence.

\begin{theorem}\label{th:norm_bound_PFASST_smoo}
  If Lemmas~\ref{lem:approx_prop_muinf} and~\ref{lem:smoo_prop_muinf_mod} hold true, then the norm of the iteration matrix $\matr{T}_{\mathrm{PFASST}}\left(\mu,k\right)$ of PFASST with $k$ modified approximative block Jacobi pre-smoothing steps can be bounded by
  \begin{align*}
      \left\lVert\matr{T}_{\mathrm{PFASST}}\left(\mu,k\right)\right\rVert \le g(k)
  \end{align*}
  with $g(k) = ck^{-\frac{1}{2}}\rightarrow 0$ as $k\rightarrow\infty$, independently of $\mu$.
\end{theorem}

\begin{proof}
  This immediately follows from Lemmas~\ref{lem:approx_prop_muinf} and~\ref{lem:smoo_prop_muinf_mod} by writing
  \begin{align*}
    \left\lVert\matr{T}_{\mathrm{PFASST}}\left(\mu,k\right)\right\rVert &= \left\lVert\left(\matr{C}^{-1} - \matr{T}_C^F \Ptimat^{-1} \matr{T}_F^C\right)\matr{C}\left(\matr{I}_{LMN} - \Phatmat^{-1}_2\matr{C}\right)\right\rVert\\
    &\le \left\lVert(\matr{C}^{-1} - \matr{T}_C^F \Ptimat^{-1} \matr{T}_F^C\right\rVert\left\lVert\matr{C}\left(\matr{I}_{LMN} - \Phatmat^{-1}_2\matr{C}\right)\right\rVert\\
    &\le c\frac{1}{\mu}\cdot\sqrt{\frac{8}{k\pi}}\mu = c\cdot\sqrt{\frac{8}{k\pi}}
  \end{align*}
  Note again that we used pre- instead of post-smoothing here to stay consistent with the standard multigrid literature.
\end{proof}

While the theoretical estimate of this theorem is much more convenient than the one of Theorem~\ref{th:norm_bound_PFASST}, practical implementations do not share this preference. 
In all tests we have done so far, using the damping factor of $2$ (or any other factor other than $1$) yields much worse convergence rates for PFASST, see also Section~\ref{sec:summary}.
We also note that the same result could have been obtained by using classical damping, i.e.~by using 
\begin{align*}
  \omega\Phatmat = \omega\left(\matr{I}_{LMN} - \mu \matr{I}_L\otimes\matr{Q}_\Delta \otimes \matr{A}\right),
\end{align*}
because the limit matrix $\matr{B}_\omega(\infty)$ is the same for both approaches.
Yet, the convergence results are even worse for this choice.

As a conclusion, the question of whether or not to use damping for the smoother in PFASST has two answers: yes, if PFASST should be a real multigrid solver and no, if PFASST should be a fast (multigrid-like) solver.

\section{Increasing the number of time-steps}
 
This last scenario, where the number of time-steps is going to infinity, is actually twofold: (1) the time interval is fixed and (2) the time interval increases with the number of time-steps $L$. 
In the first case, a fixed interval $[0,T]$ of length $T$ is divided into more and more time-steps, so that this is a special case of $\mu\rightarrow0$: here, the time-step size $\dt$ is going to $0$ as the number of time-steps goes to infinity. 
The second case, in contrast, keeps $\mu$ constant, since neither $\dt$ nor any other parameter is adapted. 
Solely the number of time-steps and therefore the length of the time interval under consideration is increasing.
Yet, for both cases the dimensions of all matrices change with $L$ and we make use of their periodic stencils in the sense of \cite{BoltenRittich2017} to find bounds for their spectral radii. In particular, Lemma A.2 of \cite{BoltenRittich2017} states that the spectral radius of an infinite block Toeplitz matrix $\matr{A}$ is equal to the essential supremum of the matrix-valued generating symbol of $\matr{A}$, providing the limit that is needed in the following.

\subsection{Fixed time interval}

For the first scenario, we again make use of the perturbation results in Lemma~\ref{lem:smoother_to0} and Theorem~\ref{th:conv_pfasst_to0}.
Following Eq.~\eqref{eq:smoo_perturbed}, we write
\begin{align*}
  \lim_{L\rightarrow\infty}\matr{T}_{\mathrm{S}}(\mu) = \lim_{L\rightarrow\infty}\matr{T}_{\mathrm{S}}(0)+ \lim_{L\rightarrow\infty}\mu\matr{D}
\end{align*}
for some perturbation matrix $\matr{D}\in\mathbb{R}^{LMN\times LMN}$, which is bounded for $\mu$ small enough (i.e.~for $L$ large enough), see the discussion leading to~\eqref{eq:smoo_perturbed}.
Now, this matrix $\matr{D}$ is bounded for all $L$, so that because $\mu\rightarrow0$ as $L\rightarrow\infty$, we have 
\begin{align*}
  \lim_{L\rightarrow\infty}\matr{T}_{\mathrm{S}}(\mu) = \lim_{L\rightarrow\infty}\matr{T}_{\mathrm{S}}(0)
\end{align*}
The matrix $\matr{T}_{\mathrm{S}}(0) = \matr{E}\otimes\matr{H}$ is a block Toeplitz matrix with periodic stencil and symbol
\begin{align*}
  \widehat{\matr{T}_{\mathrm{S}}(0)}(x) = e^{-ix}\matr{H}
\end{align*}
in the sense of \cite{BoltenRittich2017}.

Then, we have
\begin{align*}
  \rho\left(\lim_{L\rightarrow\infty}\matr{T}_{\mathrm{S}}(0)\right) = \sup_{x\in[-\pi,\pi]}\rho\left(e^{-ix}\matr{H}\right) = 1,
\end{align*}
since the eigenvalues of $e^{-ix}\matr{H}$ are $(M-1)N$-times $0$ and $M$-times $e^{-ix}$, see Theorem~\ref{th:pfasst_in_matrix_form}.
Thus, the spectral radius of $\matr{T}_{\mathrm{S}}(\mu)$ converges to $1$ for $L\rightarrow\infty$. 
Therefore, in this limit the smoother does not converge, or, more precisely, the smoother will converge slower and slower the larger $L$ becomes.
This is because for finite matrices, the spectral radius of the symbol serves as upper limit, so that the spectral radius of the iteration matrix converges to $1$ from below.
Also, this confirms the heuristic extension of Lemma~\ref{lem:smoother_to0} to infinite-sized operators, where the upper limit of the spectral radius goes to $1$ for $L\rightarrow\infty$, too.

\bigskip

For the iteration matrix $\matr{T}_{\mathrm{CGC}}$ of the coarse-grid correction, the limit matrix $\matr{T}_{\mathrm{CGC}}(0)$ is already block-diagonal, see Section~\ref{ssec:cgc_to0}.
We have also seen that the spectral radius of this matrix is at least $1$ and due to the block-diagonal structure, this does not change for $L\rightarrow\infty$.
Thus, not surprisingly, also the coarse-grid correction does not converge for $L\rightarrow\infty$.

\bigskip

Now, it seems obvious that PFASST itself will not converge, since both components alone fail to do so.
The next theorem shows that this is indeed the case, but the proof is slightly more involved.

\begin{theorem}
  For a fixed time interval with $L$ time steps and CFL number $\mu = \mu(L)$, the iteration matrix of PFASST satisfies
  \begin{align*}
    \rho\left(\lim_{L\rightarrow\infty}\matr{T}(\mu)\right) \ge 1.
  \end{align*}
\end{theorem}

\begin{proof}
  The full iteration matrix $\matr{T}(\mu)$ of PFASST can be written as 
  \begin{align*}
      \matr{T}(\mu) = \matr{T}(0) + \mathcal{O}(\mu) =\matr{E}\otimes\left(\matr{H}\left(\matr{I}_{MN}-\matr{T}_{C,Q}^F\matr{T}_{F,Q}^C\otimes\matr{T}_{C,A}^F\matr{T}_{F,A}^C\right)\right) + \mathcal{O}(\mu)
  \end{align*}
  see the discussion leading to Theorem~\ref{th:conv_pfasst_to0}.
  As before, this yields
  \begin{align*}
    \rho\left(\lim_{L\rightarrow\infty}\matr{T}(\mu)\right) = \rho\left(\lim_{L\rightarrow\infty}\matr{T}(0)\right) = \sup_{x\in[-\pi,\pi]}\rho\left(\widehat{\matr{T}(0)}(x)\right),
  \end{align*}
  following \cite{BoltenRittich2017}.

  Now, the symbol of the limit matrix $\matr{T}(0)$ is given by
  \begin{align*}
    \widehat{\matr{T}(0)}(x) = e^{-ix}\matr{H}\left(\matr{I}_{MN}-\matr{T}_{C,Q}^F\matr{T}_{F,Q}^C\otimes\matr{T}_{C,A}^F\matr{T}_{F,A}^C\right),
  \end{align*}
  which makes the computation of the eigenvalues slightly more intricate.
  Using Theorem~\ref{th:pfasst_in_matrix_form}, we write
  \begin{align*}
    \matr{H}\left(\matr{I}_{MN}-\matr{T}_{C,Q}^F\matr{T}_{F,Q}^C\otimes\matr{T}_{C,A}^F\matr{T}_{F,A}^C\right) = \matr{N}\otimes \matr{I}_N - \matr{N}\matr{T}_{C,Q}^F\matr{T}_{F,Q}^C\otimes\matr{T}_{C,A}^F\matr{T}_{F,A}^C
  \end{align*}
  and note that
  \begin{align*}
    \matr{N}\matr{T}_{C,Q}^F\matr{T}_{F,Q}^C = 
    \begin{pmatrix}
      t_{M,1} & ... & t_{M,M}\\
      \vdots &  & \vdots\\
      t_{M,1} & ... & t_{M,M},
    \end{pmatrix},
  \end{align*}
  where $t_{i,j}$ are the entries of the matrix $\matr{T}_{C,Q}^F\matr{T}_{F,Q}^C$.
  Thus, we have
  \begin{align*}
    \matr{N}\otimes \matr{I}_N - \matr{N}\matr{T}_{C,Q}^F\matr{T}_{F,Q}^C\otimes\matr{T}_{C,A}^F\matr{T}_{F,A}^C = 
    \begin{pmatrix}
      -t_{M,1}\matr{T}_{C,A}^F\matr{T}_{F,A}^C & ... &  \matr{I}_N - t_{M,M}\matr{T}_{C,A}^F\matr{T}_{F,A}^C\\
      \vdots &  & \vdots\\
      -t_{M,1}\matr{T}_{C,A}^F\matr{T}_{F,A}^C & ... & \matr{I}_N - t_{M,M}\matr{T}_{C,A}^F\matr{T}_{F,A}^C,
    \end{pmatrix}
  \end{align*}
  and the eigenvalues of this $MN\times MN$ matrix are all zero except for $N$ eigenvalues $\lambda_n$ given by the eigenvalues of
  \begin{align*}
    \matr{K} &= -\sum_{m=1}^{M-1}t_{M,m}\matr{T}_{C,A}^F\matr{T}_{F,A}^C + \matr{I}_N - t_{M,M}\matr{T}_{C,A}^F\matr{T}_{F,A}^C\\
    &= \matr{I}_N - \sum_{m=1}^{M}t_{M,m}\matr{T}_{C,A}^F\matr{T}_{F,A}^C
    = \matr{I}_N - c \matr{T}_{C,A}^F\matr{T}_{F,A}^C
  \end{align*}
  for a constant $c$ representing the sum over all $t_{M,m}$. 
  This holds since for a rank-1 matrix $\matr{B} = uv^T$ all eigenvalues are $0$ except for the eigenvalue $v^Tu$.
  In Section~\ref{ssec:cgc_to0} we already discussed that for standard Lagrangian interpolation and restriction half of the eigenvalues of multiplications like $\matr{T}_{C,A}^F\matr{T}_{F,A}^C$ are zero.
  Thus, half of the eigenvalues of $\matr{K}$ are one, so that half of the eigenvalues $\lambda_n$ are one.
  Therefore, the spectral radius of the symbol of the limit matrix can simply be bounded by
  \begin{align*}
    \rho\left(\widehat{\matr{T}(0)}(x)\right) \ge \left\lvert e^{-ix}\cdot 1\right\rvert = 1
  \end{align*}
  for all $x\in[-\pi,\pi]$, so that
  \begin{align*}
    \rho\left(\lim_{L\rightarrow\infty}\matr{T}(\mu)\right) =  \sup_{x\in[-\pi,\pi]}\rho\left(\widehat{\matr{T}(0)}(x)\right) \ge 1,
  \end{align*}
  which ends the proof.
\end{proof}
 
Therefore, PFASST itself does not converge in the limit of $L\rightarrow\infty$, if the time interval is fixed.
Also, for finite numbers of time-steps, the spectral radius of the iteration matrix is bounded by $1$, so that the spectral radius converges to $1$ from below, making PFASST slower and slower for increasing numbers of time-steps.

\subsection{Extending time interval}

In this scenario, the parameter $\mu$ does not change for $L\rightarrow\infty$.
Thus, applying the perturbation argument we frequently used in this work is not possible and fully algebraic bounds for the spectral radii of the iteration matrices do not seem possible.
However, we can make use of the results found in~\cite{doi:10.1002/nla.2110}, where a block-wise Fourier mode analysis is applied to reduce the computational effort required to compute eigenvalues and spectral radii numerically.

\bigskip

More precisely, there exist a transformation matrix $\mathcal{F}$, consisting of a permutation matrix for the Kronecker product as well as a Fourier matrix decomposing the spatial problem, such that for the smoother we have
\begin{align*}
  \matr{T}_{\mathrm{S}} = \mathcal{F}\mathrm{diag}\left(\matr{B}_1,...,\matr{B}_N\right)\mathcal{F}^{-1}
\end{align*}
for blocks
\begin{align*}
  \matr{B}_n = \matr{I}_{LM} - \left(\matr{I}_{LM} - \mu\lambda_n\matr{I}_L\otimes\matr{Q}_\Delta\right)^{-1}\left(\matr{I}_{LM} - \mu\lambda_n\matr{I}_L\otimes\matr{Q} - \matr{E}\otimes\matr{N}\right)\in\mathbb{R}^{LM\times LM}.
\end{align*}
This corresponds to the iteration matrix of the smoother for a single eigenvalue $\lambda_n(\matr{A}) = \lambda_n$ of the spatial matrix $\matr{A}$.
These blocks are block Toeplitz matrices themselves and their symbol is given by
\begin{align*}
  \widehat{\matr{B}_n}(x) &= \matr{I}_M - \left(\matr{I}_{M} - \mu\lambda_n\matr{Q}_\Delta\right)^{-1}\left(\matr{I}_{M} - \mu\lambda_n\matr{Q}\right) + e^{-ix}\matr{N} \\
  &= \left(\matr{I}_{M} - \mu\lambda_n\matr{Q}_\Delta\right)^{-1}\mu\lambda_n\left(\matr{Q}-\matr{Q}_\Delta\right) + e^{-ix}\matr{N}.
\end{align*}
Then, for $L\rightarrow\infty$ we again use \cite{BoltenRittich2017} and obtain
\begin{align*}
  \rho\left(\lim_{L\rightarrow\infty}\matr{T}_{\mathrm{S}}\right) = \max_{n=1,...,N} \rho\left(\lim_{L\rightarrow\infty}\matr{B}_n\right) = \max_{n=1,...,N} \sup_{x\in[-\pi,\pi]}\rho\left(\widehat{\matr{B}_n}(x)\right)
\end{align*}
However, although being only of size $M\times M$, it is unknown how to compute the spectral radius of these symbols for fixed $\mu$, so that numerical computation is required to find these values for a given problem.

\bigskip

Even worse, for the coarse-grid correction the blocks are at least of size $2LM\times 2LM$ due to mode mixing in space (and time, if coarsening in the nodes is applied) and they are not given by periodic stencils as for the smoother.
Thus, the results in \cite{BoltenRittich2017} cannot be applied. 
Clearly, the same is true for the full iteration matrix of PFASST and so neither the perturbation argument nor the analysis of the symbols provide conclusive bounds or limits for the spectral radius if $\mu$ is fixed.
We refer to~\cite{doi:10.1002/nla.2110,Moser2017PhD} for detailed numerical studies of these iteration matrices and their action on error vectors.

\section{Summary and discussion}\label{sec:summary}

While the parallel full approximation scheme in space and time has been used successfully with advanced space-parallel solvers on advanced HPC machines around the world, a solid mathematical analysis as well as a proof of convergence was still missing.
The algorithm in its original form is rather complex and not even straightforward to write down, posing a severe obstacle for any attempt to even formulate a conclusive theory.
Yet, with the formal equivalence to multigrid methods as shown in~\cite{doi:10.1002/nla.2110}, a mathematical framework now indeed exists which allows to use a broad range of established methods for the analysis of PFASST, at least for linear problems. 
While in~\cite{doi:10.1002/nla.2110} a detailed, semi-algebraic Fourier mode analysis revealed many interesting features and also limitations of PFASST, a rigorous convergence proof has not been given so far.
In the present paper, we used the iteration matrices of PFASST, its smoother and the coarse-grid correction to establish an asymptotic convergence theory for PFASST.
In three sections, we analyzed the convergence of PFASST for the non-stiff and the stiff limit as well as its behavior for increasing numbers of time-steps.
For small enough CFL numbers (or dispersion relations) $\mu$, we proved an upper limit for the spectral radius of PFASST's iteration matrix, which goes to zeros for smaller and smaller $\mu$ (see Theorem~\ref{th:conv_pfasst_to0}). 
In turn, if $\mu$ becomes large, we showed in Theorem~\ref{th:conv_pfasst_muinf} that the spectral radius of the iteration matrix is also bounded, but only when the parameters of the smoother are chosen appropriately. 
In this stiff limit, PFASST also satisfies the standard approximation property of linear multigrid methods, see Lemma~\ref{lem:approx_prop_muinf}, and despite the missing smoothing property, a weakened form of $\mu$-independent convergence was proven in Theorem~\ref{th:norm_bound_PFASST}.
However, in order to satisfy the smoothing property and to make PFASST a pure multigrid method, the smoother needs a damping parameter.
Then, using Reusken's Lemma, the smoothing property can be shown and fully $\mu$-independent convergence is achieved, see Theorem~\ref{th:norm_bound_PFASST_smoo}.
For all these results, we used a perturbation argument for the iteration matrices which allowed us to extract their limit matrices and to show convergence towards those.
Finally, we investigated PFASST for increasing numbers $L$ of time-steps and showed that PFASST does not converge in the limit case.
Even worse, convergence is expected to degenerate for $L$ larger and larger.
However, this applies only to a fixed time interval, i.e.\ to the case where $\mu\rightarrow0$ as $L\rightarrow\infty$. 
For an extending time interval, no analytic bound has been established.

\bigskip

The results presented here contain the first convergence proofs for PFASST for both non-stiff as well as stiff limit cases.
The lemmas and theorems of this paper are of theoretical nature and, frankly, quite technical.
Thus, the obvious question to ask is whether the results of this paper relate to actual computations with PFASST.
In order to provide first answers to this question, we used two standard test cases:

\begin{description}
  \item[\textbf{Test A}] 1D heat equation with $\nu>0$
  \begin{align*}
    u_t &= \nu\Delta u\quad\text{on } [0,1]\times[0,T],\\
    u(0,t) &= 0,\ u(1,t) = 0,\\
    u(x,0) &=\ \text{random}
  \end{align*}
  \item[\textbf{Test B}] 1D advection equation with $c>0$
  \begin{align*}
    u_t &= c\nabla u\quad\text{on } [0,1]\times[0,T],\\
    u(0,t) &= u(1,t)\\
    u(x,0) &= \sin(64\pi x)
  \end{align*}
\end{description}

For both cases and all runs we used $M=3$ Gau\ss-Radau nodes with the LU trick, $N=127$ for Test A and $N=128$ for Test B.
The time interval was fixed to $[0,T] = [0,1]$.
For both problems we used centered finite differences for the differential operators, yielding real negative eigenvalues for Test A and imaginary eigenvalues for Test B.
With Dirichlet boundary conditions in the diffusive case, the spatial matrix $\matr{A}$ is invertible and we can allow all modes to be present using a random initial guess.
In the advective case, however, $\matr{A}$ is not invertible, since two eigenvalues are zero. 
Thus, we did not use random initial data but a rather oscillatory initial sine wave.
For PFASST, we removed the initial prediction phase on the coarse level~\cite{EmmettMinion2012} and set all variables to zero at the beginning of each run, except for $u(x,0)$ (this corresponds to no spreading).
Coarsening was done in space only and we terminated the iteration when an absolute residual tolerance of $10^{-8}$ was met.
All results were generated using the \texttt{pySDC} framework and the codes can be found online, see~\cite{pySDC_asympc_conv}.

\bigskip

The first thing to analyze is whether PFASST indeed converges for the non-stiff as well as for the stiff limit.
To test this, we chose $\dt=0.25$ fixed and varied $\nu$ or $c$ such that the CFL number $\mu$ varied between $10^{-4}$ and $10^{11}$.
The results for Test A (``diffusion'') and Test B (``advection'') are shown in Figure~\ref{fig:conv_test_niter_NS3} for $k=M=3$ smoothing steps and in Figure~\ref{fig:conv_test_niter_NS2} for $k=2$ smoothing steps.
We observe how for both problems the numbers of iterations go down for $\mu$ small and $\mu$ large, at least if $k=M$.
As predicted by Theorem~\ref{th:conv_pfasst_muinf} and Lemma~\ref{lem:smoo_prop_muinf}, we need as many smoothing steps as there are quadrature nodes in order to see the number of iterations of PFASST to go down in the stiff limit while convergence in the non-stiff limit is not affected.
We also observe how the iteration counts increase for medium sizes of $\mu$ and peak at about $\mu\approx 10$, very much in line with the observations made in Figure~\ref{fig:smoother_specrad_full}.

\begin{figure}
  \centering
  \begin{subfigure}[b]{0.475\textwidth}
    \centering
    \includegraphics[width=0.95\textwidth]{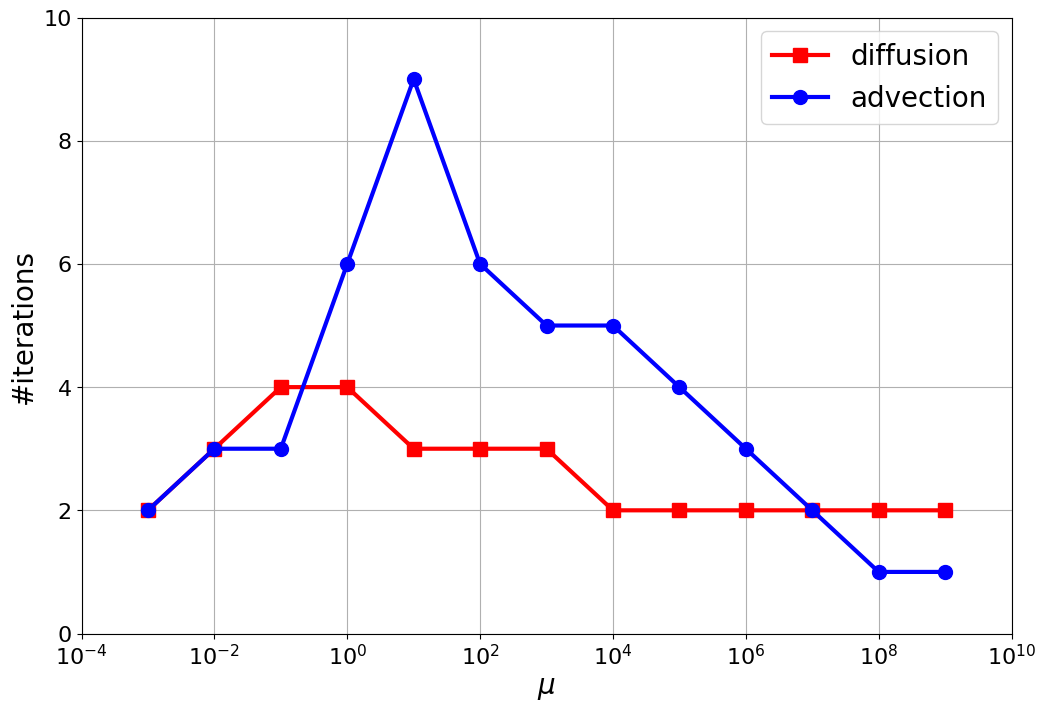}
    \caption{$3$ smoothing steps}
    \label{fig:conv_test_niter_NS3}
  \end{subfigure}
  \begin{subfigure}[b]{0.475\textwidth}
    \centering
    \includegraphics[width=0.95\textwidth]{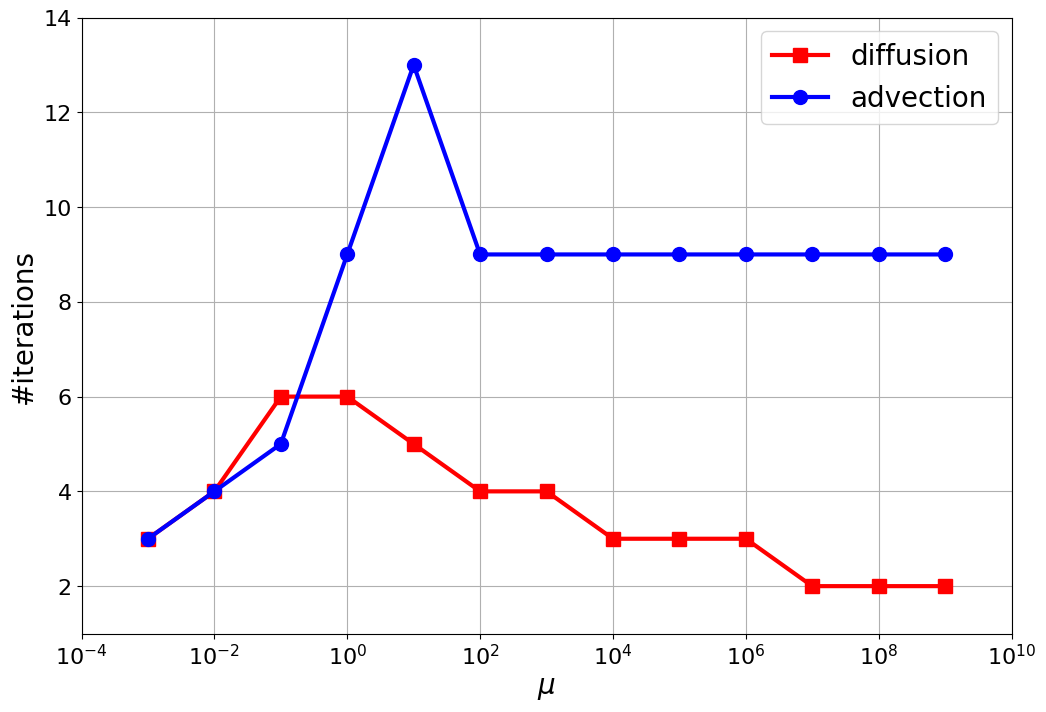}
    \caption{$2$ smoothing steps}
    \label{fig:conv_test_niter_NS2}
  \end{subfigure}
  \caption{Iteration counts for the diffusion (Test A) and advection (Test B) problem using varying parameters $\nu$ and $c$ while keeping the other parameters fixed. Left: $k=3=M$ smoothing steps, right: $k=2<M$ smoothing steps.}
  \label{fig:conv_test_niter}
\end{figure}

\bigskip

The second experiment concerns the behavior of PFASST for increasing numbers of time-steps.
For Figure~\ref{fig:conv_test_niter_Linf} we fixed $\nu=0.1$ and $c=0.1$ as well as all other parameters and used $3$ smoothing steps.
Then, we increased the number of time-steps for PFASST from $1$ to $4096$ and counted the number of iterations.
We did this for the standard, undamped smoother (``LU'') as well as for the damped smoother (``LU2'') proposed in Lemma~\ref{lem:smoo_prop_muinf_mod}, in Figure~\ref{fig:conv_test_niter_Linf_diffusion} for Test A and in Figure~\ref{fig:conv_test_niter_Linf_advection} for Test B.
This experiment shows three things: first, we see that indeed PFASST's convergence degenerates when more and more time-steps are considered, even if the time interval is fixed. 
Second, the advective case performs much worse than the diffusive case, which is to be expected from a generic time-parallel method.
Third, PFASST with damped smoothing has much worse convergence rates than the unmodified PFASST algorithm and despite being an actual multigrid solver, iteration counts increase for $L\rightarrow\infty$.
Although the increase in iteration counts is not as severe as for unmodified PFASST, there is no scenario where iteration counts are lower, making this a rather poor parallel-in-time integrator.

\begin{figure}
  \centering
  \begin{subfigure}[b]{0.475\textwidth}
    \centering
    \includegraphics[width=0.95\textwidth]{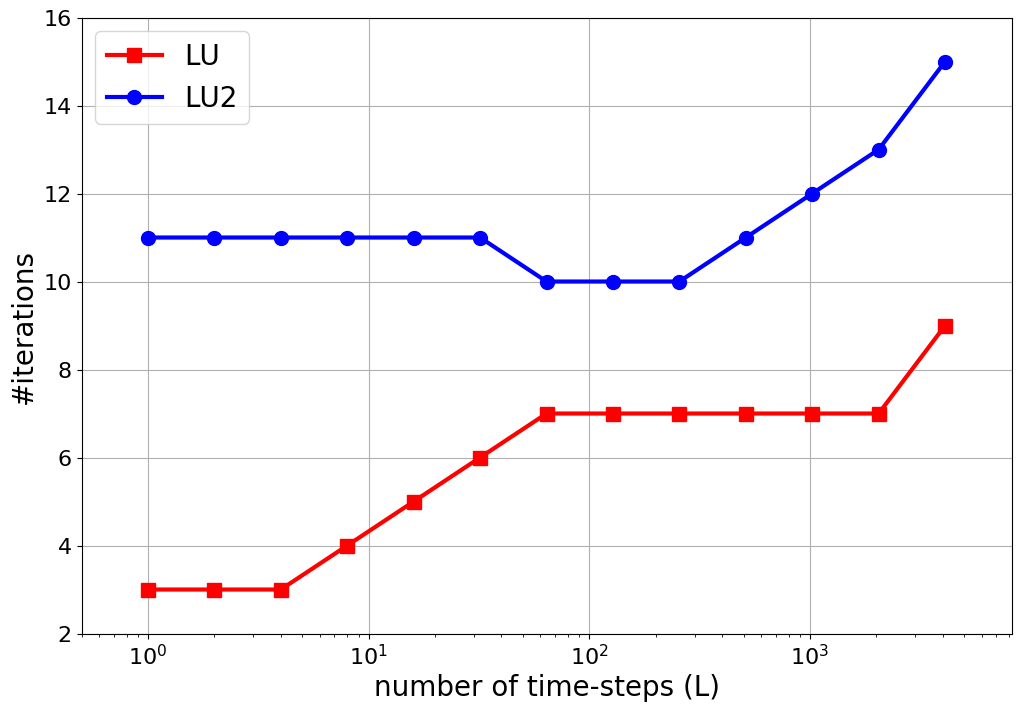}
    \caption{Diffusion}
    \label{fig:conv_test_niter_Linf_diffusion}
  \end{subfigure}
  \begin{subfigure}[b]{0.475\textwidth}
    \centering
    \includegraphics[width=0.95\textwidth]{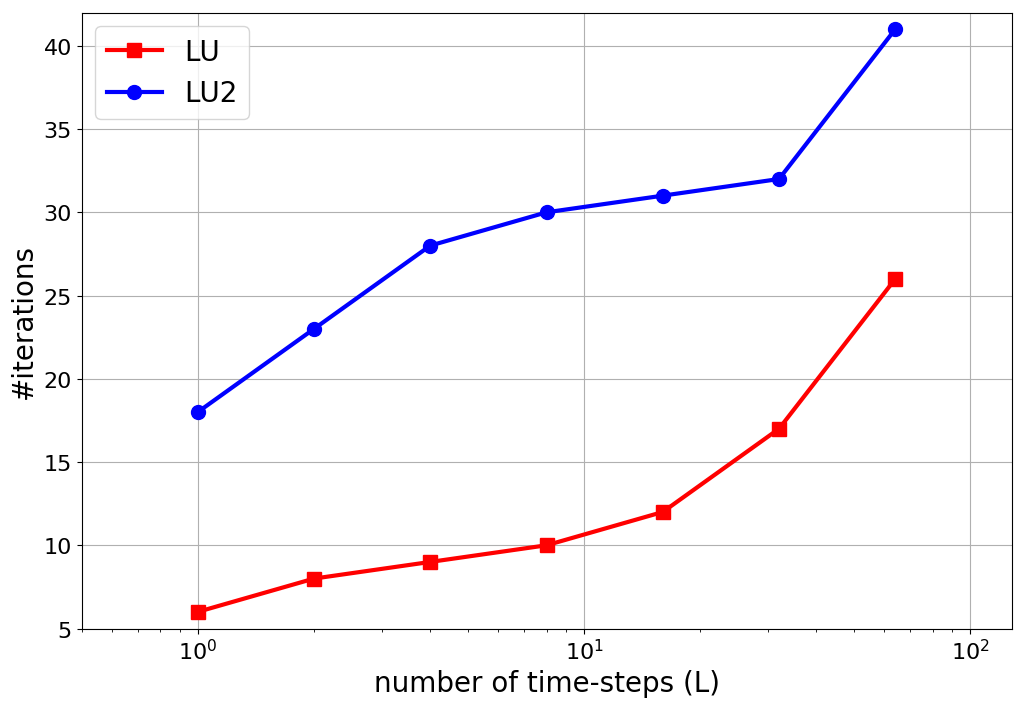}
    \caption{Advection}
    \label{fig:conv_test_niter_Linf_advection}
  \end{subfigure}
  \caption{Mean number of iterations over all time-steps for the diffusion (left) and advection (right) problem on a fixed time interval with increasing numbers of time-steps. PFASST with standard smoothing (``LU'') and damped smoothing (``LU2'') is shown.}
  \label{fig:conv_test_niter_Linf}
\end{figure}

\bigskip

In this paper, we presented the first convergence proofs for the parallel full approximation scheme in space and time, covering asymptotic cases for linear problems. 
We have seen that the results obtained here are well reflected in actual PFASST runs and can help to understand convergence behavior in realistic scenarios.
While this establishes a rather broad convergence theory, one key property of any parallel-in-time integration method cannot be investigated with this directly: parallel performance.
In the original papers on PFASST, first theoretical efficiency limits and expected speedups were already derived, but only for fixed numbers of iterations~\cite{Minion2010,EmmettMinion2012}.
Yet, the formulation of PFASST as a multigrid method now allows us not only to prove convergence but also to estimate precisely the expected parallel performance for linear problems.
We will report on this in a subsequent paper.

\bibliographystyle{wileyj} 
\bibliography{refs}

\end{document}